\documentclass{amsart}
\usepackage{amsmath,amsthm,enumerate}

\theoremstyle{theorem}
\newtheorem{theorem}{Theorem}

\newtheorem{lemma}[theorem]{Lemma}
\newtheorem{proposition}[theorem]{Proposition}
\newtheorem{conjecture}{Conjecture}

\theoremstyle{definition}

\title{Ergodic properties of matrix equilibrium states}

\begin{document}
\author{Ian D. Morris}
\begin{abstract}Given a finite irreducible set of real $d \times d$ matrices $A_1,\ldots,A_M$ and a real parameter $s>0$, there exists a unique shift-invariant equilibrium state on $\{1,\ldots,M\}^{\mathbb{N}}$ associated to $(A_1,\ldots,A_M,s)$. In this article we characterise the ergodic properties of such equilibrium states in terms of the algebraic properties of the semigroup generated by the associated matrices. We completely characterise when the equilibrium state has zero entropy, when it gives distinct Lyapunov exponents to the natural cocycle generated by $A_1,\ldots,A_M$, and when it is a Bernoulli measure. We also give a general sufficient condition for the equilibrium state to be mixing, and give an example where the equilibrium state is ergodic but not totally ergodic. Connections with a class of measures investigated by S. Kusuoka are explored in an appendix.\end{abstract}
\maketitle
\section{Introduction}
For each integer $M\geq 2$ let $\Sigma_M:=\{1,\ldots,M\}^{\mathbb{N}}$ with the infinite product topology, let $\sigma \colon \Sigma_M \to \Sigma_M$ denote the shift transformation $\sigma[(x_n)_{n=1}^\infty]:=(x_{n+1})_{n=1}^\infty$, and when $M$ is understood let $\mathcal{M}_\sigma $ denote the set of all $\sigma$-invariant Borel probabiltiy measures on $\Sigma_M$. If $x_1,\ldots,x_m \in \{1,\ldots,M\}$ are given, we define the \emph{cylinder set}
\[[x_1x_2\cdots x_m]:=\left\{y=(y_n)_{n=1}^\infty\in \Sigma_M \colon y_i=x_i\text{ for all }i=1,\ldots,m\right\}\]
which is both closed and open. It is by now classical that if $f \colon \Sigma_M \to \mathbb{R}$ is suitably regular -- for example, if
\begin{equation}\label{eq:holder}\sup\{|f(x)-f(y)|\colon x,y\in \Sigma_M\text{ and }x_1=y_1,x_2=y_2,\ldots,x_n=y_n\}=O(e^{-\gamma n})\end{equation}
for some $\gamma>0$ -- then there exists a unique measure $\mu \in \mathcal{M}_\sigma$, called the \emph{equilibrium state} of $f$, such that
\[h(f)+\int f\,d\mu = \sup_{\nu \in \mathcal{M}_\sigma } h(\nu)+\int f\,d\nu\]
where $h(\nu)$ denotes the metric entropy of the measure $\nu$ with respect to the transformation $\sigma$. These equilibrium states enjoy numerous regularity properties: they are fully supported on $\Sigma_M$, have the Bernoulli property, and satisfy the following \emph{Gibbs inequality}: for some constant $C>0$ depending only on $f$, the equilibrium measure $\mu$ of $f$ satisfies
\[C^{-1}\leq \frac{\mu([x_1\cdots x_n])}{e^{-nP(f)+\sum_{i=0}^{n-1}f(\sigma^ix)} }\leq C\]
for every $x \in \Sigma_M$ and $n\geq 1$, where $P(f)=h(\mu)+\int f\,d\mu$. Using the Gibbs property together with the Liv\v{s}ic periodic point criterion, one may also show that $f$ and $g$ as above have the same equilibrium state if and only if $f=g+h \circ \sigma - h$ for some continuous function $h$ which also satisfies \eqref{eq:holder}. The applications of the theory of these Gibbs equilibrium states are too broad and deep to even attempt to summarise here; we refer the reader only to the classic monograph \cite{PaPo90}, in which the statements made above are proved\footnote{The exception to this assertion is the Bernoulli property, which is proved in \cite{Bo74}.}.

Recently a theory has begun to emerge of Gibbs equilibrium states associated to matrix-valued, rather than real-valued, potentials. This theory is particularly motivated by its applications to multifractal analysis (\cite{Fe03,Fe09,FeLa02,Ka04}) and the study of self-affine fractals (see e.g. \cite{FeSh14,MoSh16}). Let us briefly recall some definitions. If $A_1,\ldots,A_M$ belong to the set of $d\times d$ real matrices --  which we denote by $M_d(\mathbb{R})$ -- then we define the \emph{top Lyapunov exponent} of $\mathsf{A}:=(A_1,\ldots,A_M)$ with respect to a measure $\mu \in \mathcal{M}_\sigma $ by
\[\Lambda(\mathsf{A},\mu):=\lim_{n\to\infty} \frac{1}{n}\int_{\Sigma_M}\log \|A_{x_1}\cdots A_{x_n}\|d\mu(x)=\inf_{n \geq 1}\frac{1}{n}\int_{\Sigma_M}\log \|A_{x_1}\cdots A_{x_n}\|d\mu(x).\]
Here and throughout this article, $\|\cdot\|$ denotes the Euclidean norm on $\mathbb{R}^d$ and the operator norm on $M_d(\mathbb{R})$ derived therefrom. We remark that the limit $\Lambda(\mathsf{A},\mu)$ is independent of the choice of norm used on $\mathbb{R}^d$, and in particular is unaffected by change-of-basis transformations; moreover, if $\mu$ is ergodic, then by the subadditive ergodic theorem
\[\lim_{n\to\infty} \frac{1}{n}\log \|A_{x_n}\cdots A_{x_1}\|=\Lambda(\mathsf{A},\mu)\]
for $\mu$-almost-every $x \in \Sigma_M$. 

We say that $\mathsf{A}=(A_1,\ldots,A_M)\in M_d(\mathbb{R})^M$ is \emph{irreducible} if there does not exist a linear subspace $U$ of $\mathbb{R}^d$ such that $A_iU\subseteq U$ for every $i=1,\ldots,M$ and additionally $0<\dim U <d$; otherwise we call $\mathsf{A}$ \emph{reducible}. Clearly these definitions may also be applied to sets of matrices. If $\mathsf{A}$ is reducible then we may change basis in such a way that every $A_i$ becomes block upper triangular, with the upper-left block corresponding to the action of $A_i$ on the nontrivial invariant subspace. By inductive descent it follows that if $\mathsf{A}$ is reducible then there exist an integer $k>1$ and an invertible matrix $B \in M_d(\mathbb{R})$ such that we may write
\begin{equation}\label{eq:red}B^{-1}A_iB = \left(\begin{array}{ccccc}
A_i^{(1,1)} & A_i^{(1,2)} & A_i^{(1,3)} & \cdots &A_i^{(1,k)}\\
0 & A_i^{(2,2)} & A_i^{(2,3)} & \cdots &A_i^{(2,k)}\\
0 & 0 & A_i^{(3,3)} & \cdots &A_i^{(3,k)}\\
\vdots & \vdots & \vdots  & \ddots &\vdots \\
0 & 0 & 0  & \cdots &A_i^{(k,k)}
\end{array}\right)\end{equation}
for every $i=1,\ldots,M$, where each of the $M$-tuples $\mathsf{A}^{(i)}:=(A_1^{(i,i)},A_2^{(i,i)},\ldots,A_M^{(i,i)})$ is irreducible. If $\rho(B)$ denotes the spectral radius of the matrix $B$, the reader may easily verify that in the above situation $\rho(A_{x_n}\cdots A_{x_1})=\max_{1 \leq i \leq k}\rho(A^{(i,i)}_{x_n}\cdots A^{(i,i)}_{x_1})$ for every $x_1,\ldots,x_n \in \{1,\ldots,M\}$.

 We denote the semigroup generated by the matrices $A_1,\ldots,A_M$ by $\mathcal{S}(\mathsf{A})$. Given $\mathsf{A}=(A_1,\ldots,A_M)\in M_d(\mathbb{R})^M$ and $s>0$ we define the pressure $P(\mathsf{A},s) \in [-\infty,+\infty)$ by
\begin{equation}\label{eq:pressure}P(\mathsf{A},s):=\lim_{n\to\infty} \frac{1}{n}\log\left(\sum_{i_1,\ldots,i_n=1}^M \left\|A_{x_n}\cdots A_{x_1}\right\|^s\right).\end{equation}
We observe that this limit exists by subadditivity. If $\mathsf{A}$ is irreducible and $d>1$ then $P(\mathsf{A},s)>-\infty$. 
The following corollary of the subadditive variational principle (\cite{CaFeHu08}, for an earlier result in the invertible case see \cite{Ka04}) relates $P(\mathsf{A},s)$ to the top Lyapunov exponents of $\mathsf{A}$ with respect to the various $\sigma$-invariant measures:
\begin{theorem}\label{th:subvar}
Let $\mathsf{A}=(A_1,\ldots,A_M) \in M_d(\mathbb{R})^M$ and $s>0$. Then
\[P(\mathsf{A},s)=\sup_{\mu \in \mathcal{M}_\sigma } h(\mu)+s\Lambda(\mathsf{A},\mu).\]
\end{theorem}
Let us say that $\mu$ is an \emph{equilibrium state} for $(\mathsf{A},s)$ if it attains the above supremum. We shall refer to such measures generically as \emph{matrix equilibrium states}. D.-J. Feng and A. K\"aenm\"aki \cite{FeKa11} have established the following description of the set of equilibrium states of $(\mathsf{A},s)$: 
\begin{theorem}[Feng-K\"aenm\"aki]\label{th:Gibbs}
Let $\mathsf{A}=(A_1,\ldots,A_M) \in M_d(\mathbb{R})^M$ with $M,d \geq 2$, and let $s>0$. If $\mathsf{A}$ is irreducible, then there exists a unique equilibrium state $\mu \in \mathcal{M}_\sigma $ of $(\mathsf{A},s)$. Furthermore there exists a constant $C>0$ such that for every $x_1,\ldots,x_n \in \{1,\ldots,M\}$
\begin{equation}\label{eq:Gibbs}C^{-1}\mu([x_1\cdots x_n])\leq \frac{\|A_{x_n}\cdots A_{x_1}\|^s}{e^{nP(A,s)}} \leq C\mu([x_1\cdots x_n]).\end{equation}
If $\mathsf{A}$ is not irreducible, but nonetheless $P(\mathsf{A},s)>-\infty$, let each $A_i$ be written in the form \eqref{eq:red}, and let $\mathsf{A}^{(i)}=(A_1^{(i,i)},\ldots,A_M^{(i,i)})$ for each $i=1,\ldots,\ell$. Then the set of equilibrium states of $(\mathsf{A},s)$ is nonempty, and is precisely the convex hull of the set of equlibrium states of $(\mathsf{A}^{(i)},s)$ for integers $i$ such that $P(\mathsf{A}^{(i)},s)=P(\mathsf{A},s)$.
\end{theorem}
It transpires that in the case $s=2$, these measures coincide with a class of measures investigated by S. Kusuoka \cite{Ku89}: this connection is explored in the appendix. 

Properties of matrix equilibrium states further to those listed above have to the best of the author's knowledge not yet been investigated. 
The purpose of this article is to explore the extent to which the equilibrium states of matrices enjoy the same properties as the classical equilibrium states described in the introduction. We in particular attempt to resolve the following questions:
\begin{itemize}
\item
Are matrix equilibrium states fully supported?
\item
Are matrix equilibrium states mixing with respect to the shift transformation?

\item
Do matrix equilibrium states have positive entropy?

\item
To what extent does a matrix equilibrium state allow us to reconstruct the original matrices?
\item
When are the Lyapunov exponents of a tuple of matrices (with respect to their equilibrium state) all equal to one another?
\end{itemize}
The answers to these questions vary widely in difficulty. For example, the question of full support of equilibrium states is almost trivial:
\begin{proposition}\label{th:sport}
Let $\mu$ be the unique equilibrium state of $(\mathsf{A},s)$, where $\mathsf{A}\in M_d(\mathbb{R})^M$ is irreducible. Then $\mu$ is fully supported if and only if $0 \notin \mathcal{S}(\mathsf{A})$. 
\end{proposition}
\begin{proof}
Since cylinder sets form a basis for the topology of $\Sigma_M$ it is sufficient to consider the question of whether every cylinder set has positive measure.
By the Gibbs inequality \eqref{eq:Gibbs} the cylinder $[x_1\cdots x_n]$ has zero measure if and only if $A_{x_n}\cdots A_{x_1}=0$.
\end{proof}
In this article we attempt to give the most precise possible answers to the questions listed above. We will see that the answer to each of the different questions is characterised by a slightly different linear-algebraic property, and to attempt to characterise all of the above properties in a single theorem would lead to an unreasonably long statement. We therefore address these questions in separate theorems, each of which is presented in its own chapter. In this article we are not able to give equally satisfactory answers to all of the above questions: to a significant extent this reflects the degree to which the corresponding linear-algebraic properties are understood.

We anticipate that in typical applications more will be known about the matrices $A_1,\ldots,A_M$ than the simple fact of their being irreducible. We recall that $(A_1,\ldots,A_M)$ is called \emph{strongly irreducible} if there does not exist a set $F$ which is equal to a finite union of proper nontrivial linear subspaces of $\mathbb{R}^d$ and satisfies $A_iF\subseteq F$ for every $i=1,\ldots,M$. For the reader's convenience, we note the implications of some natural hypotheses on $(A_1,\ldots,A_M)$ in the following theorem: 
\begin{theorem}
Suppose that $\mathsf{A}=(A_1,\ldots,A_M)\in M_d(\mathbb{R})^M$ is irreducible, where $M,d \geq 2$, and suppose that no $A_i$ is the zero matrix. Let $s>0$, and let $\mu$ be the equilibrium state of $(\mathsf{A},s)$. Then:
\begin{enumerate}[(i)]
\item
If $\mathsf{A}$ is strongly irreducible then $\mu$ is mixing and has nonzero entropy.
\item
If every $A_i$ is invertible then $\mu$ is fully supported and has nonzero entropy. 
\item
If $\mathcal{S}(\mathsf{A})$ contains an element whose eigenvalues are not all equal in modulus then the Lyapunov exponents of $\mathsf{A}$ with respect to $\mu$ are not all equal.
\item
If both (ii) and (iii) hold then $\mu$ is not a Bernoulli measure, and $\mu$ is not the equilibrium state of  $(\mathsf{A},t)$ for any $t \neq s$. 
\end{enumerate} 
\end{theorem}
This result follows easily from the combination of Proposition \ref{th:sport} with Theorems \ref{th:mix}, \ref{th:entropy}, \ref{th:be2} and \ref{th:whatever} below.

In the following sections we deal in turn with the properties of equilibrium states described above. At the conclusion of the article we list some problems for future research. To avoid trivialities, it will always be assumed that $M \geq 2$.


\section{Mixing}

Perhaps surprisingly, ergodic matrix equilibrium states can fail to be mixing. However, a mild additional irreducibility condition suffices to guarantee mixing:
\begin{theorem}\label{th:mix}
Let $\mathsf{A}=(A_1,\ldots,A_M) \in M_d(\mathbb{R})^M$ be irreducible, let $s>0$, and let $\mu \in \mathcal{M}_\sigma $ be the unique equilibrium state of $(\mathsf{A},s)$. Then:
\begin{enumerate}[(i)]
\item
If for every $n \geq 1$ the set of all products $A_{i_n}\cdots A_{i_1}$ such that $i_1,\ldots,i_n \in \{1,\ldots,M\}$ is irreducible, then $\mu$ is totally ergodic.
\item
If $\mu$ is totally ergodic, then it is mixing.
\end{enumerate}
\end{theorem}
\begin{proof}
(i). For every $n \geq 1$ the $M^n$-tuple $\mathsf{A}_n$ consisting of all products $A_{i_n}\cdots A_{i_1}$ with $i_1,\ldots,i_n \in \{1,\ldots,M\}$ in lexicographical order is irreducible by hypothesis, and hence $(\mathsf{A}_n,s)$ admits a unique equilibrium state by Theorem \ref{th:Gibbs}, which we interpret as a $\sigma^n$-invariant measure on $\Sigma_M$. By Theorem \ref{th:Gibbs} this measure is ergodic with respect to $\sigma^n$. Since $\mu$ is $\sigma$-invariant it is also $\sigma^n$-invariant, and is clearly also an equilibrium state for $(\mathsf{A}_n,s)$. By uniqueness it follows that for every $n\geq 1$ the measure $\mu$ must be equal to the equilibrium state of $(\mathsf{A}_n,s)$, and hence is ergodic with respect to $\sigma^n$. This demonstrates that $\mu$ is totally ergodic.

(ii). If $X=[x_1\cdots x_k]$ and $Y=[y_1\cdots y_\ell]$ are cylinder sets and $n>k$, then
\begin{eqnarray*}{\lefteqn{\mu\left(X \cap \sigma^{-n}Y\right) }}& & \\
& &= \sum_{z_1,\ldots,z_{n-k}=1}^M \mu\left([x_1\cdots x_k z_1 \cdots z_{n-k}y_1\cdots y_\ell]\right)\\
& &\leq Ce^{-(n+\ell)P(A,s)}\sum_{z_1,\ldots,z_{n-k}=1}^M \left\|A_{y_\ell}\cdots A_{y_1} A_{z_{n-k}} \cdots A_{z_{1}}A_{x_k}\cdots A_{x_1}\right\|^s\\
& &\leq Ce^{-(n+\ell)P(A,s)}\sum_{z_1,\ldots,z_{n-k}=1}^M \left\|A_{y_\ell}\cdots A_{y_1}\right\|^s\left\| A_{z_{n-k}} \cdots A_{z_{1}}\right\|^s\left\|A_{x_k}\cdots A_{x_1}\right\|^s\\
& &\leq C^4\sum_{z_1,\ldots,z_{n-k}=1}^M \mu\left([x_1\cdots x_k]\right)\mu\left([ z_1 \cdots z_{n-k}]\right)\mu\left([y_1\cdots y_\ell]\right)\\
& &=C^4 \mu\left([x_1\cdots x_k]\right)\mu\left([y_1\cdots y_\ell]\right)=C^4\mu(X)\mu(Y),\end{eqnarray*}
so in particular
\begin{equation}\label{eq:ac}\limsup_{n\to\infty} \mu\left(X\cap \sigma^{-n}Y\right) \leq C^4 \mu(X)\mu(Y)\end{equation}
for every pair of cylinder sets $X,Y \subseteq \Sigma_M$. By standard approximation arguments this inequality extends to all measurable sets $X,Y \subseteq \Sigma_M$.

To show that this property implies mixing we follow a line of argument due to D. Ornstein \cite{Or72}, which we reproduce here for the reader's convenience. We begin by showing that $\mu$ is weak-mixing. For a contradiction suppose that $f \colon \Sigma_M \to \mathbb{C}$ is a nonzero measurable function such that $f(\sigma x)=e^{2\pi i \theta}f(x)$ $\mu$-a.e, where $e^{2\pi i\theta}\neq 1$. By total ergodicity $\theta$ must be irrational. Since $|f\circ \sigma| = |f|$ almost everywhere and $\sigma$ is ergodic with respect to $\mu$, $|f|$ is constant a.e., and by multiplying by a scalar if necessary we may assume $|f|=1$ a.e. 

Let $S^1$ denote the unit circle in $\mathbb{C}$, and consider the measure $f_*\mu$ on $S^1$ defined by $(f_*\mu)(A)=\mu(f^{-1}A)$. Since $f(x)\in S^1$ a.e. this defines a Borel probability measure on $S^1$. Define $T_\theta \colon S^1 \to S^1$ by $T_\theta z =e^{2\pi i \theta}z$, and observe that $f \circ \sigma = T_\theta \circ f$ $\mu$-a.e. For every measurable $A \subseteq S^1$ we have $(f_*\mu)(T_\theta^{-1} A)=\mu(f^{-1}(T_\theta^{-1}A))=\mu(\sigma^{-1}(f^{-1}A))=\mu(f^{-1}A)=(f_*\mu)(A)$ so that $f_*\mu$ is $T_\theta$-invariant. Since $T_\theta$ is uniquely ergodic, $f_*\mu$ is Lebesgue measure on $S^1$. In particular $f_*\mu$ is not atomic, and we may choose an interval $A\subset S^1$ such that $0<C^4(f_*\mu)(A)<1$. By taking suitable rational approximations to $\theta$ we may find a sequence $(n_j)$ of natural numbers such that $T_\theta^{n_j} \to \mathrm{id}_{S^1}$ uniformly as $j \to \infty$, and it follows from this that $\limsup_{n\to\infty} (f_*\mu)(T^{-n}_\theta A \cap A) = (f_*\mu)(A)$. Hence using \eqref{eq:ac}
\begin{align*}\mu(f^{-1}A) =(f_*\mu)(A)&=\limsup_{n\to\infty}(f_*\mu)(T^{-n}_\theta A \cap A)\\
& = \limsup_{n \to \infty} \mu(\sigma^{-n}f^{-1}A \cap f^{-1}A) \leq C^4 \mu(f^{-1}A)^2<\mu(f^{-1}A),\end{align*}
a contradiction, and we conclude that $\mu$ is weak-mixing as claimed.

Consider now the sequence of measures $\nu_n$ on $\Sigma_M\times \Sigma_M$ defined by $\nu_n(X\times Y):=\mu(\sigma^{-n}X \cap Y)$. It is easy to check directly that every $\nu_n$ is $(\sigma \times \sigma)$-invariant, and it is clear that $\mu$ is mixing if and only if  the only weak-* accumulation point of $(\nu_n)_{n=1}^\infty$ is  $\mu \times \mu$. Let $\nu$ be a weak-* accumulation point of this sequence: clearly $\nu$ is $(\sigma \times \sigma)$-invariant, and by \eqref{eq:ac} $\nu$ is absolutely continuous with respect to $\mu \times \mu$. Since $\mu$  is weak-mixing with respect to $\sigma$, $\mu\times \mu$ is ergodic with respect to $\sigma \times \sigma$, so by absolute continuity we must have $\nu=\mu \times \mu$. It follows that $\lim_{n\to\infty}\nu_n = \mu \times \mu$ in the weak-* topology so that $\mu$ is mixing as required.
\end{proof}
The following example illustrates that total ergodicity can fail to hold in certain cases where the irreducibility criterion Theorem \ref{th:mix}(i) is not met. We note that while the failure of the condition in Theorem \ref{th:mix}(i) is necessary for the failure of mixing, it is not sufficient: for example, if $d=2$ and every $A_i$ is given by the same rotation through $2\pi/n$ for some $n\geq 3$, then $\mathsf{A}$ is irreducible and the condition of Theorem \ref{th:mix}(i) is not satisfied, but for every $s>0$ the equilibrium state of $(\mathsf{A},s)$ is the measure of maximal entropy which is of course mixing.
\begin{proposition}\label{pr:notmix}
Let $s>0$, and define $\mathsf{A}:=(A_1,A_2)$ where
\[A_1:=\left(\begin{array}{cc}0&2\\ 1&0\end{array} \right),\qquad A_2:=\left(\begin{array}{cc}0&1\\ 2&0\end{array} \right).\]
Then for every $s>0$, the unique equilibrium state $\mu$ of $(\mathsf{A},s)$ is not ergodic with respect to $\sigma^2$. In particular it is not mixing with respect to $\sigma$.
\end{proposition}
\begin{proof}
The pair $(A_1,A_2)$ is clearly irreducible since neither matrix has a one-dimensional invariant subspace, so $(\mathsf{A},s)$ has a unique equilibrium measure $\mu_0$ on $\Sigma_2$. Consider now the equilibrium states of the matrices
\[A_1A_1=\left(\begin{array}{cc}2&0\\0&2\end{array} \right),\quad A_1A_2=\left(\begin{array}{cc}4&0\\ 0&1\end{array} \right)\]
\[A_2A_1=\left(\begin{array}{cc}1&0\\0&4\end{array} \right),\quad A_2A_2:=\left(\begin{array}{cc}2&0\\ 0&2\end{array} \right).\]
We we shall view these equilibrium states as $\sigma^2$-invariant measures defined on $\Sigma_2$.

It follows from Theorem \ref{th:Gibbs} that an ergodic $\sigma^2$-invariant measure $\mu$ on $\Sigma_2$ is an equilibrium state of $A_1A_1$, $A_1A_2$, $A_2A_1$, $A_2A_2$ if and only if it is an equilibrium state for one of two scalar-valued potentials, the first such potential being given by the upper-left entries of those matrices each raised to the power $s$ and the second by their lower-right entries each raised to the power $s$. Elementary computations show that this gives rise to two distinct Bernoulli equilibrium states $\mu_1$, $\mu_2$, and that interchanging the matrices $A_1$ and $A_2$ with one another interchanges the two scalar-valued potentials and therefore interchanges the two distinct equilibrium states. The equilibrium state $\mu_0$ of $A_1,A_2$ is also an equilibrium state of $A_1A_1,A_1A_2,A_2A_1,A_2A_2$, but it is symmetrical with respect to the interchange of $A_1$ and $A_2$ since the pair $(A_1,A_2)$ is similar to $(A_2,A_1)$ by the transformation which interchanges the two co-ordinate axes, and Lyapunov exponents (and hence equilibrium states) are unaffected by change-of-basis transformations. The only equilibrium state of $A_1A_1,A_1A_2,A_2A_1,A_2A_2$ which is symmetrical with respect to the interchange of $A_1$ and $A_2$ is $\frac{1}{2}\mu_1+\frac{1}{2}\mu_2$, so we necessarily have $\mu_0=\frac{1}{2}\mu_1+\frac{1}{2}\mu_2$. We see that $\mu_0$ is a proper linear combination of distinct $\sigma^2$-invariant measures, and therefore $\mu_0$ is not ergodic with respect to $\sigma^2$.
\end{proof}


\section{Equilibrium states with zero entropy}\label{se:entropy}

It transpires that matrix equilibrium states have zero entropy only in highly degenerate cases which we are able to characterise precisely. In this section we shall say that $(x_1,\ldots,x_n) \in \{1,\ldots,M\}^n$ is a \emph{cyclic permutation} of $(\omega_1,\ldots,\omega_n)\in \{1,\ldots,M\}^n$ if $(x_1,\ldots,x_n)=(\omega_{i+1},\ldots,\omega_n,\omega_1,\ldots,\omega_i)$ for some $i \in \{1,\ldots,n\}$. 
\begin{theorem}\label{th:entropy}
Let $\mathsf{A}=(A_1,\ldots,A_M) \in M_d(\mathbb{R})^M$ be irreducible, let $s>0$, and let $\mu \in \mathcal{M}_\sigma $ be the unique equilibrium state of $(\mathsf{A},s)$. Then the following three conditions are equivalent:
\begin{enumerate}[(i)]
\item
The measure $\mu$ has zero entropy.
\item
The measure $\mu$ is supported on a periodic orbit of $\sigma$.
\item
There exist $n,r \geq 1$ such that  $nr=d$, symbols $\omega_1,\ldots,\omega_n \in \{1,\ldots,M\}$ and a decomposition $\mathbb{R}^d=\bigoplus_{i=1}^nR_i$ of $\mathbb{R}^d$ into $r$-dimensional subspaces such that $A_iR_j=\{0\}$ when $i \neq \omega_j$, $A_{\omega_j}R_j=R_{j+1}$ when $1 \leq j<n$, and $A_{\omega_n}R_n=R_1$. For each $i=1,\ldots,n$ the product $A_{\omega_{i-1}}\cdots A_{\omega_1}A_{\omega_n}\cdots A_{\omega_i}$ maps $R_i$ to itself bijectively, and we have $A_{x_n}\cdots A_{x_1}\neq 0$ if and only if $(x_1,\ldots,x_n)$ is a cyclic permutation of $(\omega_1,\ldots,\omega_n)$.
\end{enumerate}
\end{theorem}
We draw the reader's attention to the fact that if an equilibrium state of an irreducible tuple of matrices has zero entropy then not only is it supported on a periodic orbit, but by (iii) above the period of that orbit must divide the dimension $d$.

We remark that (iii) above is satisfied with $n=1$, $r=d=2$ if $A_1$ is an irrational rotation of $\mathbb{R}^2$ and the other $A_i's$ are zero. Less trivially we note examples such as
\[A_1=\left(\begin{array}{cc}0&1\\0&0\end{array}\right),\qquad A_2=\left(\begin{array}{cc}0&0\\1&0\end{array}\right)\]
in two dimensions, and even less trivially
\[A_1=\left(\begin{array}{cccc}0&B_1&0&0\\0&0&B_2&0\\0&0&0&0\\0&0&0&0\end{array}\right),\qquad A_2=\left(\begin{array}{cccc}0&0&0&0\\0&0&0&0\\0&0&0&B_3\\B_4&0&0&0\end{array}\right)\]
where $B_1,\ldots,B_4$ are suitably-chosen invertible block matrices; in this last case we have $A_1^2A_2^2, A_2A_1^2A_2, A_2^2A_1^2,A_1A_2^2A_1 \neq 0$ and all other products of length four are zero. It transpires that for any given periodic orbit of length $d$ on $M$ symbols one may choose an $M$-tuple of $d \times d$ matrices where that periodic orbit is the sole equilibrium state, by adapting the argument of \cite[Proposition 1.3]{Mo10}.

The proofs of most directions of implication in Theorem \ref{th:entropy} are relatively brief and self-contained, but the implication (i)$\implies$(ii) is quite involved. To shorten the proof and render it more easily digestible, we precede the proof of Theorem \ref{th:entropy} with an ancillary result. For each $p \geq 1$ let us define the \emph{$p$-radius} of $\mathsf{A}$ to be the quantity\footnote{Here we have slightly modified the definition of $\varrho_p$ used in \cite{Pr97} in order to avoid a profusion of redundant terms in subsequent calculations. The original definition includes a normalising factor $1/M^n$ inside the parenthesis and outside the summation.}
\begin{equation}\label{eq:pradius}\varrho_p(\mathsf{A}):= \lim_{n \to \infty}\left(\sum_{i_1,\ldots,i_n=1}^M\|A_{i_n}\cdots A_{i_1}\|^p\right)^{\frac{1}{np}}=e^{P(\mathsf{A},p)/p}\end{equation}
and the \emph{joint spectral radius} of $\mathsf{A}$ to be the quantity
\[\varrho_\infty(\mathsf{A}):= \lim_{n \to \infty}\max_{1 \leq i_1,\ldots,i_n\leq M}\|A_{i_n}\cdots A_{i_1}\|^\frac{1}{n}=\inf_{n \geq 1}\max_{1 \leq i_1,\ldots,i_n\leq M}\|A_{i_n}\cdots A_{i_1}\|^\frac{1}{n},\]
a definition introduced by G.-C. Rota and G. Strang (\cite{RoSt60}, reprinted in \cite{Ro03}) and investigated in depth by numerous authors (see for example \cite{Ju09}). It is a simple matter to check that both of these quantities are unaffected by a change of basis or by a change of norm on $\mathbb{R}^d$. The joint spectral radius satisfies the additional characterisation
\begin{equation}\label{eq:bwf}\varrho_\infty(\mathsf{A})= \sup_{n \geq 1} \max_{1 \leq i_1,\ldots,i_n \leq M} \rho(A_{i_n}\cdots A_{i_1})^{\frac{1}{n}},\end{equation}
a result due to M. A. Berger and Y. Wang (\cite{BeWa92}, for some interesting alternative proofs see \cite{Bo03,El95}).  We require the following:
\begin{proposition}\label{pr:reduck}
Let $\mathsf{A}=(A_1,\ldots,A_M) \in M_d(\mathbb{R})^M$ and $\ell \in \mathbb{N}$. If $\varrho_\infty(\mathsf{A})=\varrho_{2\ell}(\mathsf{A})$ then there exists $(x_1,\ldots,x_n)\in\{1,\ldots,M\}$ such that $\varrho_\infty(\mathsf{A})=\rho(A_{x_n}\cdots A_{x_1})^{1/n}$.
\end{proposition}
\begin{proof}
The trivial case $\varrho_\infty(\mathsf{A})=0$ is ignored. We first claim that without loss of generality one may take $\ell=1$. If $\ell>1$ then we may define $\mathsf{A}^{\otimes \ell}:=(A^{\otimes \ell}_1,\ldots,A^{\otimes \ell}_M) \in M_{ d^\ell}(\mathbb{R})^M$ to be the $M$-tuple of $\ell^{\mathrm{th}}$ Kronecker powers of the matrices $A_i$ (see \cite[\S4.2]{HoJo94} for definition and basic properties). Since $\|A^{\otimes \ell}\|=\|A\|^\ell$ and $A^{\otimes \ell}B^{\otimes \ell}=(AB)^{\otimes \ell}$ for every $A,B \in M_d(\mathbb{R})$ it follows by direct calculation that $\varrho_\infty(\mathsf{A}^{\otimes \ell})=\varrho_\infty(\mathsf{A})^\ell$ and $\varrho_2(\mathsf{A}^{\otimes \ell})=\varrho_{2\ell}(\mathsf{A})^{\ell}$, so we have $\varrho_2(\mathsf{A}^{\otimes \ell})=\varrho_\infty(\mathsf{A}^{\otimes \ell})$. The reader may easily see using Gelfand's formula that additionally $\rho(A^{\otimes \ell})=\rho(A)^\ell$ for every $A \in M_d(\mathbb{R})$, so if $(x_1,\ldots,x_n)$ is given such that $\varrho_\infty(\mathsf{A}^{\otimes \ell})=\rho(A_{x_n}^{\otimes \ell}\cdots A_{x_1}^{\otimes \ell})^{1/n}$ then clearly $\varrho_\infty(\mathsf{A})=\rho(A_{x_n}\cdots A_{x_1})^{1/n}$ as required.

We next claim that without loss of generality $\mathsf{A}$ may be taken to be irreducible. If $\mathsf{A}$ is not irreducible let us write $\mathsf{A}$ in the form \eqref{eq:red}; since $\rho(A_{x_n}\cdots A_{x_1})=\max_{1 \leq i \leq k}\rho(A_{x_n}^{(i,i)}\cdots A_{x_1}^{(i,i)})^{\frac{1}{n}}$ for every $(x_1,\ldots,x_n)\in\{1,\ldots,M\}^n$, it follows via \eqref{eq:bwf} that $\varrho_\infty(\mathsf{A})=\max_{1 \leq i \leq k}\varrho_\infty(\mathsf{A}^{(i)})$. Taking $i$ which achieves this maximum we find that
\[\varrho_\infty(\mathsf{A})=\varrho_\infty(\mathsf{A}^{(i)}) \leq \varrho_2(\mathsf{A}^{(i)}) \leq \varrho_2(\mathsf{A})=\varrho_\infty(\mathsf{A})\]
by appealing to the maximality of $i$, the definitions of $\varrho_\infty$ and $\varrho_2$, and the hypothesis of the proposition. We therefore have $\varrho_\infty(\mathsf{A}^{(i)})=\varrho_2(\mathsf{A}^{(i)})$ with $\mathsf{A}^{(i)}$ irreducible. If $\rho(A_{x_n}^{(i,i)}\cdots A_{x_1}^{(i,i)})^{1/n}=\varrho_\infty(\mathsf{A}^{(i)})$ then clearly
\[\varrho_\infty(\mathsf{A}) \geq \rho(A_{x_n}\cdots A_{x_1})^{1/n} \geq \rho(A_{x_n}^{(i,i)}\cdots A_{x_1}^{(i,i)})^{1/n}=\varrho_\infty(\mathsf{A}^{(i)})=\varrho_\infty(\mathsf{A})\]
and so if the conclusion of the proposition holds for the irreducible tuple $\mathsf{A}^{(i)}$ then it necessarily holds for $\mathsf{A}$.

To complete the proof it remains to establish the proposition in the case where $\ell=1$ and $\mathsf{A}$ is irreducible, for which we use an argument suggested by \cite[\S5]{Pr97} and \cite[Theorem 2]{Gu96}. By a theorem of J. Lagarias and Y. Wang (\cite[Theorem 5.1]{LaWa95}) it is sufficient to construct an inner product norm $\|\cdot\|_*$ on $\mathbb{R}^d$ such that $\|A_i\|_* \leq \varrho_\infty(\mathsf{A})$ for every $i=1,\ldots,M$. This construction comprises the remainder of the proof. Let $\langle\cdot,\cdot\rangle$ denote the Euclidean inner product on $\mathbb{R}^d$. For each integer $r \geq 1$ let us define a positive-definite symmetric bilinear form $\langle \cdot, \cdot \rangle_r$ on $\mathbb{R}^d$ by
\[\langle u,v\rangle_r:= \langle u,v\rangle + \sum_{n=1}^\infty \left(\varrho_2(\mathsf{A})^2+\frac{1}{r}\right)^{-n} \sum_{i_1,\ldots,i_n=1}^M  \langle A_{i_n}\cdots A_{i_1}u,A_{i_n}\cdots A_{i_1}v\rangle,\]
which converges by the Cauchy-Schwarz inequality together with the definition of $\varrho_2(\mathsf{A})$. Let $\|\cdot\|_r$ be the inner product norm induced by $\langle \cdot,\cdot \rangle_r$. It is easily verified that for every $i=1,\ldots,M$ and $r \geq 1$,
\[\sum_{i=1}^M \| A_iv\|^2_r \leq \left(\varrho_2(\mathsf{A})^2+\frac{1}{r}\right)\|v\|^2_r \]
for every $v \in \mathbb{R}^d$. Let us normalise each $\langle \cdot, \cdot \rangle_r$ so that $\max\{\|v\|_r \colon \|v\|\leq 1\}=1$. By the Cauchy-Schwarz inequality each of the forms $\langle \cdot,\cdot \rangle_r$ is then $K$-Lipschitz in each variable with respect to the Euclidean distance on the Euclidean $K$-ball about the origin. By the Arzel\'a-Ascoli theorem we may thus choose a subsequence $(r_j)_{j=1}^\infty$ such that $\langle \cdot, \cdot \rangle_{r_j}$ converges uniformly on compact subsets of $(\mathbb{R}^d)^2$ to a positive semidefinite bilinear form $\langle \cdot,\cdot \rangle_*$ such that $\max\{\|v\|_*\colon \|v\|\leq 1\}=1$ and $\sum_{i=1}^M \|A_iv\|^2_*\leq \varrho_2(\mathsf{A})^2\|v\|^2_*$ for every $v \in \mathbb{R}^d$.

We claim that $\langle\cdot,\cdot\rangle_*$ is positive definite, which is to say that $\|\cdot\|_*$ is a norm and not merely a seminorm. If $\|\cdot\|_*$ is zero on a nontrivial linear subspace $U$ of $\mathbb{R}^d$ then the inequality $\sum_{i=1}^M \|A_iv\|^2_*\leq \varrho_2(\mathsf{A})^2\|v\|^2_*$ implies that $A_iU \subseteq U$ for every $i=1,\ldots,M$; but by irreducibility this implies $U=\mathbb{R}^d$, which contradicts  $\max\{\|v\|_*\colon \|v\|\leq 1\}=1$. We conclude that $\|\cdot\|_*$ is an inner product norm on $\mathbb{R}^d$. We have
\begin{align*}\max_{1 \leq i \leq M}\|A_i\|_* &= \max_{\|v\|_*=1} \max_{1 \leq i \leq M}\|A_iv\|_* \\
&\leq \max_{\|v\|_*=1}\left(\sum_{i=1}^M \|A_iv\|_*^2\right)^{\frac{1}{2}}\\
&=\max_{\|v\|_*=1} \varrho_2(\mathsf{A})\|v\|_* = \varrho_\infty(\mathsf{A})\end{align*}
so that $\|A_i\|_*\leq \varrho_\infty(\mathsf{A})$ for every $i=1,\ldots,M$, and since $\|\cdot\|_*$ is an inner product norm we may apply \cite[Theorem 5.1]{LaWa95} to establish the existence of the desired product $A_{x_n}\cdots A_{x_1}$.
\end{proof}

\begin{proof}[Proof of Theorem \ref{th:entropy}]
The implication (ii)$\implies$(i) is trivial, so we shall begin by establishing (i)$\implies$(ii). Let us therefore assume that the equilibrium measure $\mu$ of $(\mathsf{A},s)$ satisfies (i), and let $2\ell >s$ be an even integer. We claim that $\mu$ is also the equilibrium state of $(\mathsf{A},2\ell)$. Indeed, we have
\begin{align*}\sup_{\nu \in \mathcal{M}_\sigma } h(\nu)+2\ell\Lambda(\mathsf{A},\nu)& =\frac{2\ell}{s}\left(\sup_{\nu \in \mathcal{M}_\sigma } \frac{s}{2\ell}h(\nu)+s\Lambda(\mathsf{A},\nu) \right) \\
&\leq \frac{2\ell}{s}\left(\sup_{\nu \in \mathcal{M}_\sigma } h(\nu)+s\Lambda(\mathsf{A},\nu) \right) \\
&=\frac{2\ell}{s}\left(h(\mu)+s\Lambda(\mathsf{A},\mu) \right) \\
&=h(\mu)+2\ell\Lambda(\mathsf{A},\mu)\end{align*}
using the fact that $s/2\ell<1$, the fact that $\mu$ is an equilibrium state for $(\mathsf{A},s)$, and finally the fact that $h(\mu)=0$. This inequality demonstrates that $\mu$ is the equilibrium state of $(\mathsf{A},2\ell)$ as claimed. 

Now, comparing \eqref{eq:pressure} and \eqref{eq:pradius} we have
\[\varrho_{2\ell}(\mathsf{A})= \lim_{n \to \infty}\left(\sum_{i_1,\ldots,i_n=1}^M\|A_{i_1}\cdots A_{i_n}\|^{2\ell}\right)^{\frac{1}{2n\ell}}=e^{P(\mathsf{A},2\ell)/2\ell},\]
and since $P(\mathsf{A},2\ell)=h(\mu)+2\ell\Lambda(\mathsf{A},\mu)=2\ell\Lambda(\mathsf{A},\mu)$ we have $\varrho_{2\ell}(\mathsf{A})= e^{\Lambda(\mathsf{A},\mu)}$. By \cite[Proposition 2.2]{Mo13} on the other hand we have
\[\log\varrho_\infty(\mathsf{A})=\sup_{\nu \in \mathcal{M}_\sigma} \Lambda(\mathsf{A},\nu),\]
and since
\begin{align*}\sup_{\nu \in \mathcal{M}_\sigma} \Lambda(\mathsf{A},\nu) &\leq \sup_{\nu \in \mathcal{M}_\sigma}(2\ell)^{-1}h(\nu)+ \Lambda(\mathsf{A},\nu)\\
&=(2\ell)^{-1}P(\mathsf{A},2\ell)\\
&=\Lambda(\mathsf{A},\mu)\leq \sup_{\nu \in \mathcal{M}_\sigma} \Lambda(\mathsf{A},\nu)\end{align*}
it follows that $\log \varrho_\infty(\mathsf{A})=\Lambda(\mathsf{A},\mu)$. We therefore have $\varrho_{2\ell}(\mathsf{A})=\varrho_\infty(\mathsf{A})$, and by Proposition \ref{pr:reduck}
there exist an integer $n \geq 1$ and finite sequence $(x_1,\ldots,x_n) \in \{1,\ldots,M\}^n$ such that $\varrho_\infty(\mathsf{A})=\rho(A_{x_n}\cdots A_{x_1})^{\frac{1}{n}}$. Define $z \in \Sigma_M$ by $z_{qn+r}=x_r$ for every $q \geq 1$ and $r \in \{1,\ldots,n\}$ so that $\sigma^nz=z$, and let $\mu':=\frac{1}{n}\sum_{i=0}^{n-1}\delta_{\sigma^iz} \in \mathcal{M}_{\sigma} $. We have
\[\Lambda(\mathsf{A},\mu')=\lim_{m \to \infty} \frac{1}{m} \int \log \|A_{y_n}\cdots A_{y_1}\|d\mu'(y) = \frac{1}{n}\log\rho(A_{x_n}\cdots A_{x_1})\]
using Gelfand's formula, so
\[\Lambda(\mathsf{A},\mu')=\frac{1}{n}\log \rho(A_{x_n}\cdots A_{x_1}) = \log\varrho_\infty(\mathsf{A}) =\Lambda(\mathsf{A},\mu).\]
Hence
\[h(\mu')+s\Lambda(\mathsf{A},\mu')=s\Lambda(\mathsf{A},\mu')= s\Lambda(\mathsf{A},\mu) = h(\mu)+s\Lambda(\mathsf{A},\mu)=P(\mathsf{A},s)\]
so that $\mu'$ is an equilibrium state of $(\mathsf{A},s)$. Since only one such equilibrium state exists we have $\mu'=\mu$. We conclude that $\mu$ is supported on a periodic orbit as claimed, and this completes the proof of the equivalence of (i) and (ii).

To prove (iii)$\implies$(ii) we argue as follows. Let $\omega_1,\ldots,\omega_n$ be as in (iii) and let $z \in \Sigma_M$ be defined by $z_{qn+r}:=\omega_r$ for every $q \geq1$ and $0\leq r<n$. Clearly $\sigma^nz=z$. It is straightforward to see that if $y \notin \{z,\sigma z,\ldots,\sigma^{n-1}z\}$, then for large enough $m$ the finite sequence $(y_1,\ldots,y_m)$ contains a finite subword $(y_k,y_{k+1},\ldots,y_{k+n-1})$ which is not equal to a cyclic permutation of $(\omega_1,\ldots,\omega_n)$. In particular we have $A_{y_{k+n-1}}\cdots A_{y_k}=0$ using (iii), and therefore $A_{y_m}\cdots A_{y_1}=0$. It follows from the Gibbs inequality \eqref{eq:Gibbs} that $\mu([y_1\cdots y_m])=0$, and since $[y_1\cdots y_m]$ is an open neighbourhood of $y$, the point $y$ is not in the support of $\mu$. Since $y \notin \{z,\sigma z,\ldots,\sigma^{n-1}z\}$ was arbitrary it follows that the support of $\mu$ is a subset of $\{z,\sigma z\ldots,\sigma^{n-1}z\}$, and $\mu$ is supported on a periodic orbit as claimed. This completes the proof of the implication (iii)$\implies$(ii). 

It remains only to prove (ii)$\implies$(iii). Suppose that $\omega=\sigma^n\omega \in \Sigma_M$ is periodic with period $n$ and that $\mu=\frac{1}{n}\sum_{i=0}^{n-1}\delta_{\sigma^i\omega}$ is the unique equilibrium state of the irreducible matrices $A_1,\ldots,A_M$. Since $\mu([x_1\cdots x_n])\neq 0$ if and only if $(x_1,\ldots,x_n)=(\omega_{i-1}\cdots \omega_1\omega_n\cdots \omega_i)$ for some $i=1,\ldots,n$, it follows from the Gibbs property that $\|A_{x_n}\cdots A_{x_1}\|\neq 0$ if and only if the same condition holds. Since $\mu$ is an equilibrium state its top Lyapunov exponent cannot equal $-\infty$, so we have $\rho(A_{\omega_n}\cdots A_{\omega_1})>0$.

We recall that if $A \colon \mathbb{R}^d \to \mathbb{R}^d$ is a linear map and $v \in \mathbb{R}^d$, then $A^kv=0$ for some $k \geq 1$ if and only if $A^dv=0$. In particular $A^d\mathbb{R}^d$ is an $A$-invariant subspace of $\mathbb{R}^d$ on which $A$ acts bijectively. Define $r$ to be the rank of the product $(A_{\omega_n}\cdots A_{\omega_1})^d$. Since this product has nonzero spectral radius, this rank is nonzero, and since the rank of a product $AB$ is always equal to the rank of the product $BA$, the rank of every product $(A_{\omega_{i-1}}\cdots A_{\omega_1}A_{\omega_n}\cdots A_{\omega_i})^d$ is also equal to $r$. For each $i=1,\ldots,n$ define $R_i:=(A_{\omega_{i-1}}\cdots A_{\omega_1}A_{\omega_n}\cdots A_{\omega_i})^d\mathbb{R}^d$ and note that each $R_i$ has dimension $r$ and that $A_{\omega_{i-1}}\cdots A_{\omega_1}A_{\omega_n}\cdots A_{\omega_i}$ is a bijection from $R_i$ to itself. It is clear from the definition that $A_{\omega_i}R_i\subseteq R_{i+1}$ when $1 \leq i <n$ and that $A_{\omega_n}R_n\subseteq R_1$, and since $A_{\omega_{i-1}}\cdots A_{\omega_1}A_{\omega_n}\cdots A_{\omega_i}$ is a bijection from $R_i$ to itself we must in fact have $A_{\omega_i}R_i= R_{i+1}$ when $1 \leq i <n$ and $A_{\omega_n}R_n= R_1$.

Let $1 \leq j \leq M$ and $1 \leq i \leq n$. If $j \neq \omega_i$ then $A_j A_{\omega_{i-1}}\cdots A_{\omega_1}A_{\omega_n}\cdots A_{\omega_i}$ is zero, and therefore $A_jR_i=\{0\}$. If $j=\omega_i$ then we of course have $A_jR_i=R_{i+1}$ if $i<n$, or $A_jR_i=R_1$ otherwise. For each $k=1,\ldots,n$ let $V_k$ denote the span of the union of $R_1,\ldots,R_k$. We observe in particular that $A_jV_n \subseteq V_n$ for every $j=1,\ldots,M$, and since $V_n$ contains the $r$-dimensional space $R_1$, it follows that the dimension of $V_n$ is nonzero so that by irreducibility $V_n=\mathbb{R}^d$.

We claim that $\dim V_k = kr$ for every $k=1,\ldots,n$. For $k=1$ this is clear. Given that $\dim V_k=kr$, to show that $\dim V_{k+1}=(k+1)r$ it suffices to show that $V_k \cap R_{k+1}=0$. When $i \neq j$ the product
\[(A_{\omega_{i-1}}\cdots A_{\omega_1}A_{\omega_n}\cdots A_{\omega_i})(A_{\omega_{j-1}}\cdots A_{\omega_1}A_{\omega_n}\cdots A_{\omega_{j}})\]
is zero, so $A_{\omega_{i-1}}\cdots A_{\omega_1}A_{\omega_n}\cdots A_{\omega_i}R_j=\{0\}$ when $i \neq j$ by the definition of $R_j$. We also of course have $A_{\omega_{k}}\cdots A_{\omega_1} A_{\omega_n}\cdots A_{\omega_{k+1}}R_{k+1}=R_{k+1}$. It follows from these observations that the product $A_{\omega_{k}}\cdots A_{\omega_1} A_{\omega_n}\cdots A_{\omega_{k+1}}$ is identically zero on $V_k$ and has trivial kernel on $R_{k+1}$. We deduce that $V_k$ contains only the trivial element of $R_{k+1}$ and therefore $V_k \cap R_{k+1}=\{0\}$ as required. We conclude by induction that $d=\dim V_n =nr$, and it follows that $\mathbb{R}^d = \bigoplus_{i=1}^n R_i$ as claimed. This completes the proof of (ii)$\implies$(iii) and hence completes the proof of the theorem.\end{proof}


\section{Characterisation by equilibrium states}\label{se:char}

In this section we investigate when $(\mathsf{A},s)$ and $(\mathsf{B},t)$ can have the same equilibrium state when either $\mathsf{A}$ is distinct from $\mathsf{B}$, $s$ is distinct from $t$, or both.

Clearly if two $M$-tuples $\mathsf{A},\mathsf{B} \in M_d(\mathbb{R})^M$ are conjugate to one another by a simultaneous change-of-basis transformation then for each $s>0$ the equilibrium state of $(\mathsf{A},s)$ equals that of $(\mathsf{B},s)$. The converse is easily seen to be false, since two finite irreducible subsets of $O(2)$ may fail to be conjugate in this fashion, but both will have the same equilibrium states since all of their respective products will have norm $1$. Nonetheless the following result shows that each equilibrium state of a tuple $\mathsf{A}$ completely determines, and is determined by, the spectral radius of every product of the elements of $\mathsf{A}$, and moreover contains enough information to determine completely the top Lyapunov exponent of $\mathsf{A}$ with respect to all invariant measures. 

\begin{theorem}\label{th:char}
Let $\mathsf{A}=(A_1,\ldots,A_M) \in M_{d_1}(\mathbb{R})^M$, $\mathsf{B}=(B_1,\ldots,B_M) \in M_{d_2}(\mathbb{R})^M$ be irreducible, where $d_1,d_2 \geq 1$. Let $s,t>0$, and suppose that $P(\mathsf{A},s),P(\mathsf{B},t)>-\infty$. Then the following are equivalent:
\begin{enumerate}[(i)]
\item
The equilibrium states of $(\mathsf{A},s)$ and $(\mathsf{B},t)$ are identical.
\item
There is a constant $C>0$ such that
\[C^{-1}e^{-n\mathsf{P}(\mathsf{A},s)}\|A_{i_n}\cdots A_{i_1}\|^s \leq e^{-n\mathsf{P}(\mathsf{B},t)}\|B_{i_n}\cdots B_{i_1}\|^t \leq Ce^{-n\mathsf{P}(\mathsf{A},s)}\|A_{i_n}\cdots A_{i_1}\|^s\]
for every $(i_1,\ldots,i_n) \in \{1,\ldots,M\}^n$.
\item
We have 
\[e^{-n\mathsf{P}(\mathsf{A},s)}\rho(A_{i_n}\cdots A_{i_1})^s =e^{-n\mathsf{P}(\mathsf{B},t)}\rho(B_{i_n}\cdots B_{i_1})^t\]
for every $(i_1,\ldots,i_n) \in \{1,\ldots,M\}^n$.
\item
There is a constant $\lambda \in \mathbb{R}$ such that  $s\Lambda(\mathsf{A},\nu)=\lambda+t\Lambda(\mathsf{B},\nu)$ for every $\nu \in \mathcal{M}_\sigma $.
\end{enumerate}
\end{theorem}
\begin{proof}
The implication (iv)$\implies$(i) is a trivial consequence of the definition of equilibrium state given after Theorem \ref{th:subvar}; the implication (i)$\implies$(ii) is a simple corollary of the Gibbs inequality; and the implication (ii)$\implies$(iii) follows from Gelfand's formula. To prove (iii)$\implies$(iv) we argue as follows. Define $\lambda:=P(\mathsf{A},s)-P(\mathsf{B},t)$ so that
\begin{equation}\label{eq:whatever}\frac{s}{n}\log \rho(A_{i_n}\cdots A_{i_1}) = \lambda + \frac{t}{n}\log \rho(B_{i_n}\cdots B_{i_1}) \end{equation}
for every $(i_1,\ldots,i_n) \in \{1,\ldots,M\}^n$. If $\nu \in \mathcal{M}_\sigma $ is ergodic then by \cite[Theorem 1.6]{Mo12} together with the subadditive ergodic theorem we have for $\mu$-a.e. $x \in \Sigma_M$
\[\Lambda(\mathsf{A},\nu)=\lim_{n \to \infty} \frac{1}{n}\log\|A_{x_n}\cdots A_{x_1}\|=\limsup_{n\to\infty} \frac{1}{n}\log\rho(A_{x_n}\cdots A_{x_1})\]
and 
\[\Lambda(\mathsf{B},\nu)=\lim_{n \to \infty} \frac{1}{n}\log\|B_{x_n}\cdots B_{x_1}\|=\limsup_{n\to\infty} \frac{1}{n}\log\rho(B_{x_n}\cdots B_{x_1}).\]
In view of \eqref{eq:whatever} we deduce that $s\Lambda(\mathsf{A},\nu) = \lambda + t\Lambda(\mathsf{B},\nu)$. We conclude that (iv) is satisfied for all ergodic measures $\nu$, and the general case follows by a standard ergodic decomposition argument.
\end{proof}
While Theorem \ref{th:char} gives insight into when $(\mathsf{A},s)$ and $(\mathsf{B},s)$ can have the same equilibrium state, its implications for the comparison between the equilibrium states of $(\mathsf{A},s_1)$ and $(\mathsf{A},s_2)$ are at least as interesting. We note the following:
\begin{theorem}\label{th:cks}
Let $\mathsf{A}=(A_1,\ldots,A_M) \in M_{d}(\mathbb{R})^M$ be irreducible with $d \geq 2$. Then the following are equivalent:
\begin{enumerate}[(i)]
\item
There exist distinct $s_1,s_2>0$ such that the equilibrium states of $(\mathsf{A},s_1)$ and $(\mathsf{A},s_2)$ are identical.
\item
There is a number $\lambda \in \mathbb{R}$ such that every element of $\mathcal{S}(e^{-\lambda}\mathsf{A})$ has spectral radius either $1$ or $0$.
\item
There exist $a \geq 0$, $b \in \mathbb{R}$ such that $P(\mathsf{A},s)=a+bs$ for all $s>0$.
\item
The equilibrium state of $(\mathsf{A},s)$ is the same for every $s>0$.
\end{enumerate}
\end{theorem}
\begin{proof}
The implication (iv)$\implies$(i) is trivial. If (i) holds then by Theorem \ref{th:char} we have
\[-nP(\mathsf{A},s_1)+ s_1\log\rho(A_{x_n}\cdots A_{x_1})=-nP(\mathsf{A},s_2)+ s_2\log\rho(A_{x_n}\cdots A_{x_1})\]
for every $(x_1,\ldots,x_n)$, so that in each case either $\rho(A_{x_n}\cdots A_{x_1})=0$ or
\[(s_1-s_2)\log\rho(A_{x_n}\cdots A_{x_1}) =n(P(\mathsf{A},s_1)-P(\mathsf{A},s_2)).\]
Taking $\lambda:=(P(\mathsf{A},s_1)-P(\mathsf{A},s_2))/(s_1-s_2)$ it follows that $e^{-\lambda n}\rho(A_{x_n}\cdots A_{x_1}) \in \{0,1\}$ for every $(x_1,\ldots,x_n)$, and this proves (ii).

If (ii) holds then by an argument similar to the proof of Theorem \ref{th:char} it follows that for every $\nu \in \mathcal{M}_\sigma$ the top Lyapunov exponent $\Lambda(\mathsf{A},\nu)$ is either $\lambda$ or $-\infty$, and therefore
\[P(\mathsf{A},s)=s\lambda + \sup\{h(\nu)\colon \nu \in \mathcal{M}_\sigma \text{ and }\Lambda(\mathsf{A},\nu)\neq -\infty\}.\]
We note that the supremum is over a nonempty set since $P(\mathsf{A},s)>-\infty$ by irreducibility, and this establishes (iii).

If (iii) holds, let $s_1,s_2>0$ and let $\mu_1,\mu_2$ be the equilibrium states of $(\mathsf{A},s_1)$ and $(\mathsf{A},s_2)$ respectively. We have
\begin{align*}P\left(\mathsf{A},\frac{s_1+s_2}{2}\right)&=\frac{1}{2}P(\mathsf{A},s_1)+\frac{1}{2}P(\mathsf{A},s_2)\\
&=\frac{1}{2}\left(h(\mu_1)+s_1\Lambda(\mathsf{A},\mu_1)\right)+\frac{1}{2}\left(h(\mu_1)+s_2\Lambda(\mathsf{A},\mu_2)\right)\\
&=h\left(\frac{1}{2}\left(\mu_1+\mu_2\right)\right)+ \Lambda\left(\mathsf{A},\frac{1}{2}\left(\mu_1+\mu_2\right)\right)\end{align*}
so that $\frac{1}{2}(\mu_1+\mu_2)$ is an equilibrium state for $(\mathsf{A},\frac{s_1+s_2}{2})$. By Theorem \ref{th:Gibbs} the equilibrium state of $(\mathsf{A},\frac{s_1+s_2}{2})$ is unique and ergodic, so $\frac{1}{2}(\mu_1+\mu_2)$ is an ergodic measure, which is only possible if $\mu_1=\mu_2$. Since $s_1,s_2>0$ were arbitrary this proves (iv) and completes the proof of the theorem. 
\end{proof}
Without additional assumptions on the matrices $A_i$ we are unable to shed further light on exactly when the situation of Theorem \ref{th:cks} occurs, since the structure of matrix semigroups in which every element has spectral radius either $1$ or $0$ does not appear to have been investigated. It is also unclear in general exactly which measures can arise as the equilibrium state in this case. However, at the end of the following section we will see that when the matrices $A_i$ are invertible, the situation of Theorem \ref{th:cks} can be neatly characterised.


\section{Bernoulli measures and Lyapunov exponents }

We recall that a \emph{Bernoulli measure} on $\Sigma_M$ is a probability measure of the form $\nu^{\mathbb{N}}$ where $\nu$ is a probability measure on $\{1,\ldots,M\}$. 
The following result characterises those cases in which a matrix equilibrium state is a Bernoulli measure:
\begin{theorem}\label{th:be1}
Let $\mathsf{A}=(A_1,\ldots,A_M) \in M_d(\mathbb{R})^M$ be irreducible, let $s>0$, and let $\mu \in \mathcal{M}_\sigma $ be the unique equilibrium state of $(\mathsf{A},s)$. Then the measure $\mu$ is a Bernoulli measure if any only if for every $B_1,B_2 \in \mathcal{S}(A)$ we have $\rho(B_1B_2)=\rho(B_1)\rho(B_2)$.
\end{theorem}
Before proving the theorem we require the following lemma, which is of a standard type but is sharper than typical formulations (compare e.g.  \cite[Proposition 2.8]{Fe09}, \cite[Lemma 3.1]{Wi02}).
\begin{lemma}\label{le:ire}
Let $\mathsf{A}=(A_1,\ldots,A_M)\in M_d(\mathbb{R})^M$ be irreducible. Then there exists $\delta>0$ such that for every $B_1,B_2 \in M_d(\mathbb{R})$ we may find $(i_1,\ldots,i_k)$ with $0 \leq k<d$ such that $\|B_1A_{i_1}\cdots A_{i_k}B_2\|\geq \delta \|B_1\|\cdot\|B_2\|$. (Here we interpret a product of length zero as being the identity matrix.)
\end{lemma}
\begin{proof}
Clearly we may restrict our attention to the case in which $\|B_1\|=\|B_2\|=1$, the remaining cases following by homogeneity. By compactness it suffices to show that we may always choose a product $A=A_{i_1}\cdots A_{i_k}$ such that $B_1AB_2$ is nonzero. Fix $B_1$ and $B_2$ and let $v$ be a nonzero vector in the image of $B_2$. For each $k=0,\ldots,d-1$ let $V_k$ denote the space spanned by all vectors of the form $A_{i_1}\cdots A_{i_m}v$ where $0 \leq m \leq k$. We have $\dim V_0=1$ and $\dim V_{k+1} \geq \dim V_k$ for every $k$, and clearly $A_iV_k \subseteq V_{k+1}$ when $0 \leq k<d-1$. Since $1 =\dim V_0 \leq \dim V_k \leq d$ for every $k$, we have $\dim V_{k+1}=\dim V_k$ for some $k$ by the pigeonhole principle, so $V_k$ is an $\mathsf{A}$-invariant subspace with nonzero dimension and therefore is equal to $\mathbb{R}^d$. In particular $V_{d-1}=\mathbb{R}^d$ and so $V_{d-1}$ contains a vector which is not in the kernel of $B_1$. This implies that there exists $A=A_{i_1}\cdots A_{i_k}$ with $0 \leq k <d$ such that $B_1Av \neq 0$, and therefore $B_1AB_2\neq 0$ as desired.
\end{proof}

\begin{proof}[Proof of Theorem \ref{th:be1}]
Suppose that $\mu$ is a Bernoulli measure. If $(x_1,\ldots,x_n)$ and $(y_1,\ldots,y_n)$ are arbitrary permutations of each other then
\begin{align*}\|A_{x_n}\cdots A_{x_1}\|^s &\leq Ce^{nP(\mathsf{A},s)} \mu([x_1\cdots x_n])\\
&=Ce^{nP(\mathsf{A},s)} \mu([y_1\cdots y_n])\leq C^2\|A_{y_n}\cdots A_{y_1}\|^s\end{align*}
by the Gibbs inequality \eqref{eq:Gibbs}. By this principle it follows that if $B_1,B_2 \in \mathcal{S}(\mathsf{A})$ then for every $n \geq 1$
\[\|(B_1B_2)^n\|\leq C^{\frac{2}{s}} \|B_1^nB_2^n\| \leq C^{\frac{2}{s}} \|B_1^n\|\cdot\|B_2^n\|\]
and therefore $\rho(B_1B_2)\leq \rho(B_1)\rho(B_2)$ by Gelfand's formula. On the other hand, by the preceding lemma there exist a constant $\delta>0$ and a sequence of matrices $B_3^{(n)}=A_{i_1(n)}\cdots A_{i_{k_n(n)}}$ such that for every $n \geq 1$
\[\left\|B_1^nB_3^{(n)}B_2^n\right\| \geq \delta \|B_1^n\|\cdot \|B_2^n\|.\]
Since the matrices $B_3^{(n)}$ are all drawn from the same finite set there exists in particular a fixed matrix $B_3 \in \mathcal{S}(\mathsf{A}) \cup \{\mathrm{Id}\}$ such that $\|B_1^nB_3B_2^n\| \geq \delta \|B_1^n\|\cdot \|B_2^n\|$ for infinitely many $n$. Hence
\[\delta \|B_1^n\|\cdot \|B_2^n\| \leq \|B_1^nB_3B_2^n\| \leq C^{\frac{2}{s}} \left\|B_3(B_1B_2)^n\right\| \leq C^{\frac{2}{s}}\|B_3\|\cdot \left\|(B_1B_2)^n\right\|\]
for infinitely many $n$, and therefore $\rho(B_1)\rho(B_2)\leq \rho(B_1B_2)$ by Gelfand's formula again. This proves the first direction of the theorem.

To prove the second direction, we note that the one-dimensional matrices $B_1$ up to $B_M$ defined by $B_i:=\rho(A_i)$ satisfy $\rho(A_{i_1}\cdots A_{i_n})=\rho(B_{i_1}\cdots B_{i_n})$ for every $i_1,\ldots,i_n \in \{1,\ldots,M\}^n$ for every $n\geq 1$, and it follows by Theorem \ref{th:char} that $\mu$ is the equilibrium state of $((B_1,\ldots,B_M),s)$; but the equilibrium state of a set of one-dimensional matrices is a Bernoulli measure.
\end{proof}
A semigroup $\mathcal{S}$ with the property that $\rho(B_1B_2)=\rho(B_1)\rho(B_2)$ for all $B_1,B_2\in\mathcal{S}$ is sometimes said to have \emph{multiplicative spectral radius}. Semigroups of non-invertible matrices with this property are at present not completely understood (though see \cite{OmRa97,PrVo14}).  An example of a finitely-generated semigroup of non-invertible matrices in which the spectral radius is multiplicative is provided by the matrices
\[A_1:=\left(\begin{array}{cc}1&1\\0&0\end{array}\right), \quad A_2:=\left(\begin{array}{cc}1&-1\\0&0\end{array}\right),\quad A_3:=\left(\begin{array}{cc}0&0\\1&1\end{array}\right),\quad A_4:=\left(\begin{array}{cc}0&0\\-1&1\end{array}\right).\]
The reader may verify directly that $\{\pm A_1,\ldots, \pm A_4\}$ is a semigroup in which every element has spectral radius $1$. Moreover since $A_1$ and $A_2$ have an eigenbasis consisting of the horizontal axis and the negative diagonal, and $A_3$ and $A_4$ have an eigenbasis consisting of the vertical axis and the positive diagonal, this set is irreducible.

However, semigroups of \emph{invertible} matrices with multiplicative spectral radius have been completely characterised in \cite{PrVo14}, and this allows us to present a second theorem on Bernoulli measures as equilibrium states. For each $A \in M_d(\mathbb{R})$ we recall that the \emph{singular values} of $A$, denoted $\alpha_1(A),\ldots,\alpha_d(A)$, are the non-negative square roots of the eigenvalues of the positive semidefinite matrix $A^TA$, listed in decreasing order with repetition in the case of multiplicity. We note the identity $\alpha_1(A)\cdots \alpha_d(A)=|\det A|$ for every $A \in M_d(\mathbb{R})$. Given $\mu \in \mathcal{M}_\sigma $ and $\mathsf{A} \in M_d(\mathbb{R})^M$ we define the \emph{Lyapunov exponents} of $\mathsf{A}$ with respect to $\mu$ to be the numbers
\[\lambda_i(\mathsf{A},\mu):=\lim_{n\to\infty} \frac{1}{n}\int_{\Sigma_M}\log\alpha_i(A_{x_n}\cdots A_{x_1})d\mu(x).\]
(The existence of these limits follows from the inequality
\[\prod_{i=1}^k \alpha_i(AB)\leq \prod_{i=1}^k \alpha_i(A)\alpha_i(B)\]
for every $A,B \in M_d(\mathbb{R})$ and $k =1,\ldots,d$, which guarantees the existence of
\[\sum_{i=1}^k \lambda_i(\mathsf{A},\mu)=\lim_{n\to\infty}\frac{1}{n}\int_{\Sigma_M}\sum_{i=1}^k\log\alpha_i(A_{x_n}\cdots A_{x_1})d\mu(x)\]
by subadditivity.)  We observe that $\lambda_1(\mathsf{A},\mu)$ coincides with the top Lyapunov exponent $\Lambda(\mathsf{A},\mu)$ defined in section 1. A result related to Theorem \ref{th:be2} is presented in \cite[\S3]{MoSh16} in the context of equilibrium states of the singular value function (see \S\ref{se:end} for the definition).
\begin{theorem}\label{th:be2}
Let $\mathsf{A}=(A_1,\ldots,A_M) \in M_d(\mathbb{R})^M$ be irreducible and suppose that none of the matrices $A_i$ is equal to the zero matrix. Let  $s>0$ and let $\mu$ the equilibrium state of $(\mathsf{A},s)$. Then the following three conditions are equivalent:
\begin{enumerate}[(i)]
\item
The measure $\mu$ is a fully-supported Bernoulli measure, and each of the matrices $A_1,\ldots,A_M$ is invertible;
\item
The Lyapunov exponents
\[\lambda_i(A,\mu):=\lim_{n\to\infty} \frac{1}{n}\int_{\Sigma_M} \log\alpha_i(A_{x_n}\cdots A_{x_1}) d\mu(x) \]
satisfy $\lambda_1(A,\mu)=\lambda_2(A,\mu)=\cdots =\lambda_d(A,\mu)$;
\item
There exists $B \in GL_d(\mathbb{R})$ such that $|\det A_i|^{-1/d} B^{-1}A_iB \in O(d)$ for every $i=1,\ldots,M$.
\end{enumerate}
\end{theorem}
\begin{proof}
If (i) holds then by Theorem \ref{th:be1} we have $\rho(B_1B_2)=\rho(B_1)\rho(B_2)$ for every $B_1,B_2 \in \mathcal{S}(\mathsf{A})$. The semigroup generated by the matrices $\rho(A_i)^{-1}A_i$ thus consists entirely of matrices with spectral radius equal to $1$. A theorem of V. Protasov and A. Voynov (\cite[Theorem 2]{PrVo14}) implies that this semigroup is conjugate to a subsemigroup of $O(d)$ by a suitable change-of-basis matrix. It follows that $\rho(A_i)=|\det A_i|^{1/d}$ for every $i=1,\ldots,M$, and we have obtained (iii). If (iii) holds then by a suitable change of basis we may take $B$ to be the identity. By the definition of $O(d)$ every $A_{i_1}\cdots A_{i_n}$ satisfies $(A_{i_1}\cdots A_{i_n})^TA_{i_1}\cdots A_{i_n}=(\det A_{i_1}\cdots A_{i_n})^{2/d} \mathrm{Id}$ and so the singular values of $A_{i_1}\cdots A_{i_n}$ are all equal to $|\det A_{i_1}\cdots A_{i_n}|^{1/d}$. It follows that the Lyapunov exponents $\lambda_i(A,\nu)$ are independent of $i$ for every $\nu \in \mathcal{M}_\sigma $, and we obtain (ii).

Let us now prove (ii)$\implies$(i). We have
\begin{align*}h(\mu)+\frac{s}{d}\int \log|\det A_{x_1}|d\mu(x) &= h(\mu)+\lim_{n\to\infty} \frac{s}{nd}\int \log |\det A_{x_n}\cdots A_{x_1}|d\mu(x)\\
& =  h(\mu)+\lim_{n\to\infty} \frac{1}{n}\int \frac{s}{d}\sum_{i=1}^d \log \alpha_i( A_{x_n}\cdots A_{x_1})d\mu(x)\\
&=h(\mu)+\frac{s}{d}\sum_{i=1}^d \lambda_i(\mathsf{A},\mu)=h(\mu)+s\lambda_1(\mathsf{A},\mu)\end{align*}
using the identity $|\det B|=\prod_{i=1}^d \alpha_i(B)$ and the hypothesis (ii), and since for every $\nu \in\mathcal{M}_\sigma $
\begin{align*}h(\mu)+s\lambda_1(\mathsf{A},\mu)&\geq h(\nu)+s\lambda_1(\mathsf{A},\nu)\\
& \geq h(\nu)+\frac{s}{d}\sum_{i=1}^d \lambda_i(\mathsf{A},\nu)=h(\nu)+\frac{s}{d} \int \log|\det A_{x_1}|d\nu(x) \end{align*}
using the elementary identity $\lambda_1(\mathsf{A},\nu)=\max_{1 \leq i \leq d}\lambda_i(\mathsf{A},\nu)$, it follows that
\[h(\mu)+\frac{s}{d}\int \log|\det A_{x_1}|d\mu(x) =\sup_{\nu \in \mathcal{M}_\sigma } h(\nu)+\frac{s}{d}\int \log|\det A_{x_1}|d\nu(x)\]
so that $\mu$ is the equilibrium state of the scalar potential $x\mapsto \frac{s}{d}\log |\det A_{x_1}|$. Since this potential depends only on the first co-ordinate of $x$, $\mu$ must be the  unique Bernoulli measure which satisfies
\[\mu([k])=  \frac{|\det A_k|^{s/d}}{e^{P(\mathsf{A},s)}}\]
for each $k=1,\ldots,M$. Using the Gibbs inequality we deduce
\[|\det A_k|^{s/d} \geq C^{-1}\|A_k\|^s>0\]
for each $k=1,\ldots,M$ and it follows that every $A_k$ is invertible as required, completing the proof of the theorem.
\end{proof}

We may now prove the result mentioned at the end of the previous section, which describes completely the situation in which the matrices $A_i$ are all invertible and in which the equilibrium state of $(\mathsf{A},s)$ is independent of $s$:
\begin{theorem}\label{th:whatever}
Let $\mathsf{A}=(A_1,\ldots,A_M)\in M_d(\mathbb{R})^M$ be irreducible, and suppose that every $A_i$ is invertible. Then the following are equivalent:
\begin{enumerate}[(i)]
\item
The equilibrium state of $(\mathsf{A},s)$ is the same for every $s$.
\item
There exists $s>0$ such that the equilibrium state of $(\mathsf{A},s)$ is the measure of maximal entropy on $\Sigma_M$.
\item
There exist $\lambda \in \mathbb{R}$ and invertible $B \in M_d(\mathbb{R})$ such that $e^{-\lambda}B^{-1}A_iB \in O(d)$ for every $i=1,\ldots,M$.
\end{enumerate}
\end{theorem}
\begin{proof}
If (iii) holds then it is clear that $\Lambda(\mathsf{A},\nu)=\lambda$ for every $\nu \in \mathcal{M}_\sigma$, and it follows from the variational principle that for every $s>0$ the equilibrium state of $(\mathsf{A},s)$ is the measure of maximal entropy. This implies both (i) and (ii).

 If (i) holds then by Theorem \ref{th:cks} there is a number $\lambda$ such that $\mathcal{S}(e^{-\lambda}\mathsf{A})$ consists only of matrices with spectral radius $1$ or $0$, and by invertibility they must all have spectral radius $1$. It follows that $\rho(A_{x_n}\cdots A_{x_1})=e^{\lambda n}$ for every $x_1,\ldots,x_n \in \{1,\ldots,M\}$ and every $n \geq 1$. By an argument similar to Theorem \ref{th:char} it follows that $\Lambda(\mathsf{A},\nu)=\lambda$ for every $\nu \in \mathcal{M}_\sigma$. It follows by the variational principle that for every $s>0$ the measure of maximal entropy is the equilibrium state of $(\mathsf{A},s)$, and this implies (ii).
 
 If (ii) holds then since the equilibrium measure of  $(\mathsf{A},s)$ is a Bernoulli measure, it follows by by Theorem \ref{th:be2} that there is an invertible matrix $B \in M_d(\mathbb{R})$ such that $|\det A_i|^{1/d}B^{-1}A_iB \in O(d)$ for every $i=1,\ldots,M$. The measure $\mu$ is therefore the unique Bernoulli measure such that
 \[\mu([i])=\frac{|\det A_i|^{\frac{s}{d}}}{e^{P(\mathsf{A},s)}}\]
 for every $i=1,\ldots,M$.  Since $\mu$ is the measure of maximal entropy it follows that $|\det A_i|^{1/d}$ takes a value independent of $i$, which we call $e^{-\lambda}$. This proves (iii) and completes the proof of the theorem.
\end{proof}


\section{Directions for future research}\label{se:end}
\subsection{Mixing properties}
We have not been able to completely characterise when matrix equilibrium states are mixing, and this is a natural problem for future research. This depends in particular on the problem of characterising exactly when the $M^n$-tuple of products $A_{i_1}\cdots A_{i_n}$ is reducible while the original $M$-tuple $(A_1,\ldots,A_M)$ is irreducible. In two dimensions this problem seems likely to be accessible by elementary arguments, but in higher dimensions the situation is less clear. We however make the following conjecture:
\begin{conjecture}
Let $\mathsf{A}\in M_d(\mathbb{R})^M$ be irreducible. Then:
\begin{enumerate}[(i)]
\item
If for some $s>0$ the equilibrium state of $(\mathsf{A},s)$ is ergodic with respect to $\sigma^d$, then it is mixing with respect to $\sigma$.
\item
If there exists $s>0$ such that the equilibrium state of $(\mathsf{A},s)$ is mixing with respect to $\sigma$, then for \emph{every} $s>0$ the equilibrium state of $(\mathsf{A},s)$ is mixing with respect to $\sigma$.
\end{enumerate}
\end{conjecture}
We have so far paid no attention to stronger forms of mixing. In view of the quantitative mixing properties of the equilibrium states of real-valued potentials -- and of results in the case $s=2$ which are described in the appendix --  it is natural to ask whether cylinder sets mix exponentially rapidly with respect to matrix equilibrium states. 
We have also not investigated the Kolmogorov or Bernoulli properties of matrix equilibrium states (or strictly speaking, of their natural extensions). We pose the following problem for future researchers:
\begin{conjecture}\label{cj:conjbern}
The natural extension of every totally ergodic matrix equilibrium state has the Bernoulli property.
\end{conjecture} 

\subsection{When is the equilibrium state independent of $s$?}

In Theorem \ref{th:cks} it was shown that an irreducible $M$-tuple $\mathsf{A}\in M_d(\mathbb{R})$ has the same equilibrium state $\mu$ for every $s>0$ if and only if there is a constant $\lambda$ such that every element of $\mathcal{S}(e^{-\lambda} \mathsf{A})$ has spectral radius either $0$ or $1$. In the case of invertible matrices this situation was characterised completely in Theorem \ref{th:whatever}. Other than in the invertible case --  in which case a spectral radius of $0$ of course cannot occur -- no attempt at characterising irreducible matrix semigroups in which every matrix has spectral radius $0$ or $1$ seems to have been made in the literature. Indeed, even the characterisation of irreducible matrix semigroups in which every element has spectral radius $1$ is not yet complete (see \cite{OmRa97,PrVo14}). It is interesting to ask what necessary or sufficient conditions can be found for $(A_1,\ldots,A_M)$ to generate a semigroup with this property. It is also not clear exactly which measures can arise as the equilibrium state of such a tuple, although by Theorem \ref{th:whatever} this set of measures includes the measure of maximal entropy, and by adapting the examples mentioned in \S\ref{se:entropy} it may be seen to also include periodic orbits whose length divides $d$.

\subsection{Criteria for $\overline{d}$-continuity}

The $\overline{d}$-metric, or Ornstein metric, is a metric on the set of $\sigma$-invariant measures on $\{1,\ldots,M\}^{\mathbb{Z}}$ which captures dynamical features of measures more effectively than the weak-* topology. For example, with respect to the $\overline{d}$-metric entropy is continuous, and the set of measures with the Bernoulli property, the set of mixing measures and the set of ergodic measures are all closed sets. A definition of the $\overline{d}$-metric and proofs of some of its basic properties may be found in \cite{Gl03,Sh96}. Z. Coelho and A. Quas have shown that the equilibrium states of real-valued potentials $f \colon \Sigma_M \to \mathbb{R}$ which satisfy \eqref{eq:holder} depend $\overline{d}$-continuously on the potential in an appropriate topology on the space of potentials \cite{CoQu98}. It would be of interest to determine exactly when matrix equilibrium states (or strictly speaking, their natural extensions) depend $\overline{d}$-continuously on the matrix entries.

Using the results in this article it is possible to show that the equilibrium states of irreducible matrix sets do not always depend $\overline{d}$-continuously on the matrix entries: one may verify easily that for every $\varepsilon>0$ the pair $\mathsf{A}_\varepsilon:=(A_1,A_2)$ given by
\[A_1:=\left(\begin{array}{cc}0&2\\ 1&0 \end{array} \right),\qquad A_2:=\left(\begin{array}{cc}\varepsilon &1\\ 2&0\end{array} \right)\]
is strongly irreducible, and therefore by Theorem \ref{th:mix} the unique equilibrium state of $(\mathsf{A}_\varepsilon,1)$ is mixing. On the other hand by Proposition \ref{pr:notmix}, the limit as $\varepsilon \to 0$ is a pair $\mathsf{A}_0$ such that the unique equilibrium state of $(\mathsf{A}_0,1)$ is \emph{not} mixing. This in particular implies that the said equilibrium states do not vary $\overline{d}$-continuously in the limit $\varepsilon \to 0$. We nonetheless make the following conjectures:
\begin{conjecture}
Let $\mathsf{A} \in M_d(\mathbb{R})^M$ be irreducible. Then the unique equilibrium state $\mu_s$ of $(\mathsf{A},s)$ depends $\overline{d}$-continuously on $s$.
\end{conjecture}

\begin{conjecture}\label{cj:conjdbar}
For each irreducible $\mathsf{B} \in M_d(\mathbb{R})^M$ and $t>0$ let $\mu_{\mathsf{B},t}$ denote the unique equilibrium state of $(\mathsf{B},t)$. Suppose that $\mathsf{A} \in M_d(\mathbb{R})^M$ is irreducible and that $\mu_{\mathsf{A},s}$ is mixing. Then $(\mathsf{B},t)\mapsto \mu_{\mathsf{B},t}$  is $\overline{d}$-continuous at $(\mathsf{A},s)$.
\end{conjecture}
We remark that if Conjecture \ref{cj:conjbern} is true then Conjecture \ref{cj:conjdbar} follows directly. To see this we note that 
\begin{equation}\label{eq:converge}\lim_{(\mathsf{B},t)\to (\mathsf{A},s)}\mu_{\mathsf{B},t}=\mu_{\mathsf{A},s}\end{equation}
in the weak-* topology by the following argument. By weak-* compactness, for every sequence of pairs $(\mathsf{B},t)$ converging to $(\mathsf{A},s)$ the sequence of measures $\mu_{\mathsf{B},t}$ has an accumulation point. To prove \eqref{eq:converge} it suffices to show that the only possible accumulation point is $\mu_{\mathsf{A},s}$. It is shown in \cite{FeSh14,Mo15} that the pressure $P(\mathsf{B},t)$ depends continuously on $(\mathsf{B},t)$. By this fact together with the upper semicontinuity of entropy and of Lyapunov exponents it follows that any accumulation point of $\mu_{\mathsf{B},t}$ as $(\mathsf{B},t) \to (\mathsf{A},s)$ must be an equilibrium state of $(\mathsf{A},s)$, which by uniqueness must equal $\mu_{\mathsf{A},s}$. The convergence \eqref{eq:converge} follows.
One may also show that necessarily
\[\lim_{(\mathsf{B},t)\to (\mathsf{A},s)}h(\mu_{\mathsf{B},t})=h(\mu_{\mathsf{A},s}),\]
since
\[\limsup_{(\mathsf{B},t)\to (\mathsf{A},s)}h(\mu_{\mathsf{B},t})\leq h(\mu_{\mathsf{A},s})\]
by the upper semi-continuity of entropy, and 
\[\liminf_{(\mathsf{B},t)\to (\mathsf{A},s)}h(\mu_{\mathsf{B},t}) =\liminf_{(\mathsf{B},t)\to (\mathsf{A},s)}\left(P(\mathsf{B},t)-t\Lambda(\mathsf{B},t)\right) \geq P(\mathsf{A},s)-s\Lambda(\mathsf{A},s) = h(\mu_{\mathsf{A},s})\]
by the continuity of pressure and the upper semi-continuity of Lyapunov exponents. Since $\mu_{\mathsf{A},s}$ by hypothesis has the Bernoulli property, these two limits together imply convergence in $\overline{d}$ by the finitely-determined property of Bernoulli processes (see e.g. \cite[\S IV.2]{Sh96}).

\subsection{Equilibrium states of the singular value function}
A further problem for future research is the extension of the results of this article to the equilibrium states of the \emph{singular value function}, defined as follows. If $A \in M_d(\mathbb{R})$ and $0<s<d$ then we define
\[\varphi^s(A):=\alpha_1(A)\cdots \alpha_{\lfloor s\rfloor}(A) \alpha_{\lceil s \rceil}(A)^{s-\lfloor s\rfloor},\]
where $\alpha_0(A):=1$, and for $s\geq d$ we define $\varphi^s(A):=|\det A|^{s/d}$. One may show that $\varphi^s(AB) \leq \varphi^s(A)\varphi^s(B)$ for every $A,B \in M_d(\mathbb{R})$ and $s>0$, and that $\varphi^s \colon M_d(\mathbb{R})\to \mathbb{R}$ is upper semi-continuous. Given $\mathsf{A}=(A_1,\ldots,A_M)\in M_d(\mathbb{R})^M$, an equilibrium state of $\varphi^s$ is defined to be a measure $\mu \in \mathcal{M}_\sigma $ which maximises the quantity
\[h(\mu)+ \lim_{n\to\infty}\frac{1}{n} \int \log\varphi^s(A_{x_n}\cdots A_{x_1})d\mu(x),\]
and the existence of at least one such equilibrium state for given $\mathsf{A}$ and $s$ is guaranteed by an upper semi-continuity argument.
When $0<s \leq 1$ we simply have $\varphi^s(A)=\|A\|^s$, and when $d-1 \leq s <d$ we have
\[\varphi^s(A)=\left\||\det A|^{\frac{s+1-d}{d-s}}A^{\wedge(d-1)}\right\|^{d-s},\]
so in these parameter ranges the equilibrium states of $\varphi^s$ may be understood in terms of the norm equilibrium states studied in the present article. (Indeed, this reduction is implicitly exploited in the article \cite{MoSh16}.) In view of applications of these equilibrium states in the dimension theory of self-affine fractals (see e.g. \cite{BaRa15,FaKe15,FeSh14,KaRe14,MoSh16}) it would be of interest to be able to extend the results of this article to the equilibrium states of the singular value function in the parameter range $1<s<d-1$. However, at present no analogue of Theorem \ref{th:Gibbs} for these equilibrium states is known. 

\section{Acknowledgements}
The problems investigated in this article were suggested in part by recent joint research of the author and Pablo Shmerkin \cite{MoSh16}. The author thanks Pablo Shmerkin for numerous helpful conversations, particularly with regard to Theorems \ref{th:mix}, \ref{th:be1} and \ref{th:be2}. The author was supported by the Engineering and Physical Sciences Research Council (grant number EP/L026953/1). 

In respect of RCUK policies on publicly-funded research data, the author notes that no research data were generated in the course of this research.

\appendix
\section{Kusuoka measures and the case $s=2$}

 In this appendix we will show that in the case where $s=2$ both the pressure and the equilibrium state admit simple closed-form expressions, the latter in terms of the \emph{Kusuoka measures} defined by S. Kusuoka \cite{Ku89} which have been the subject of recent research \cite{BeHoSt14,JoObPo15,StTs10}.

Recall that a matrix $B \in M_d(\mathbb{R})$ is called \emph{positive semidefinite} if and only if $B^T=B$ and $\langle Bu,u\rangle \geq 0$ for every $u \in \mathbb{R}^d$. The matrix is called \emph{positive definite} if additionally $\langle Bu,u\rangle>0$ when the vector $u$ is nonzero. The matrix $B$ is positive semidefinite (resp. positive definite) if and only if it has the form $B=U^TU$ for some matrix $U \in M_d(\mathbb{R})$ (resp. for some invertible matrix $U \in M_d(\mathbb{R})$), and this matrix $U$ can be taken to be upper triangular with non-negative diagonal entries. Every eigenvalue of a positive semidefinite matrix is real and non-negative. For our purposes a \emph{cone} in a finite-dimensional real vector space $V$ will be a closed convex set $K \subset V$ with nonempty interior such that $\lambda K = K$ for every real $\lambda>0$, and such that $K \cap (-K)=\{0\}$. Let $M_d^S(\mathbb{R})$ denote the vector space of symmetric real $d \times d$ matrices. The reader may trivially verify that the set of all positive semidefinite $d \times d$ matrices is a cone in $M_d^S(\mathbb{R})$ and that the set of all positive definite $d \times d$ matrices is the interior of that cone.

In order to define Kusuoka measures we require the following statement:
\begin{proposition}\label{pr:ku}
Let $\mathsf{A}=(A_1,\ldots,A_M)$ be irreducible, and define two linear maps $L_{\mathsf{A}}, \hat{L}_{\mathsf{A}} \colon M_d^S(\mathbb{R}) \to M_d^S(\mathbb{R})$ by
\[L_{\mathsf{A}}B:=\sum_{i=1}^M A_i^TBA_i,\qquad \hat{L}_{\mathsf{A}}B:=\sum_{i=1}^M A_iBA_i^T.\]
Then 
\begin{equation}\label{eq:prids}e^{P(\mathsf{A},2)}=\rho\left(L_{\mathsf{A}}\right)=\rho\left(\hat{L}_{\mathsf{A}}\right)=\rho\left(\sum_{i=1}^M A_i^{\otimes 2}\right),\end{equation}
and there exist unique positive definite matrices $Q,\hat{Q} \in M_d^S(\mathbb{R})$ such that $\mathrm{tr}\,Q\hat{Q}=1$ and
\begin{equation}\label{eq:kuid}L_{\mathsf{A}}Q=e^{P(\mathsf{A},2)}Q,\qquad \hat{L}_{\mathsf{A}}\hat{Q}=e^{P(\mathsf{A},2)}\hat{Q}.\end{equation}
\end{proposition}
{\bf{Remark.}} 
Proposition \ref{pr:ku} summarises results of Kusuoka, K.-S. Lau, J. Wang and Yu. V. Protasov. The existence of the eigenmatrices \eqref{eq:kuid} was proved by Kusuoka in \cite{Ku89}. The identity of the eigenvalue $\rho(L_{\mathsf{A}})$ with the pressure $e^{P(\mathsf{A},2)}$ does not seem to have been previously noticed in the literature on Kusuoka measures, nor in the literature on matrix equilibrium states. This identity was however noted by Protasov \cite{Pr97} who attributed it to Lau and Wang \cite{LaWa95b}. Formul{\ae} involving expressions of the form $\rho(\sum_{i=1}^M A_i^{\otimes k})$ have become a staple of joint spectral radius research (see e.g. \cite{BlNe05,Pr10}) but we have found it surprisingly difficult to find a proof of the precise statement \eqref{eq:prids} in the literature. It being relatively economical to simply prove the above statements all at once ourselves, we therefore take the liberty of presenting the proofs of these results in the above expository proposition.
\begin{proof}
Ignoring for the moment the normalisation condition $\mathrm{tr}\,Q\hat{Q}=1$ which will be treated at the end of the proof, we note that it is sufficient to prove only those statements concerned with $L_{\mathsf{A}}$. Let $\mathsf{A}^T:=(A_1^T,\ldots,A_M^T)$. Clearly $\hat{L}_{\mathsf{A}}=L_{\mathsf{A}^T}$ and $P(\mathsf{A},2)=P(\mathsf{A}^T,2)$, so the claimed results for $\hat{L}_{\mathsf{A}}$ will follow from those for $L_{\mathsf{A}}$ if it is shown that $\mathsf{A}^T$ is also irreducible. To prove this we note that $\mathsf{A}$ has a nontrivial invariant proper subspace if and only if there exist nonzero $u,v \in \mathbb{R}^d$ such that $\langle A_iu,v\rangle=0$ for every $i=1,\ldots,M$, and this property clearly holds for $\mathsf{A}$ iff it holds for $\mathsf{A}^T$.

Let us therefore prove the asserted properties of $L_{\mathsf{A}}$. Let $\mathcal{K}\subset M_d^S(\mathbb{R})$ denote the cone of positive semidefinite matrices. We assert that $\mathcal{L}:=\sum_{k=0}^{d-1}L_{\mathsf{A}}^k$ maps $\mathcal{K}\setminus \{0\}$ into the interior of $\mathcal{K}$. To see this let $B \in \mathcal{K}$ and $u \in \mathbb{R}^d$ be nonzero: we must show that $\langle \mathcal{L}Bu,u\rangle>0$, which implies that $\mathcal{L}B$ is positive definite. Let $B=U^TU$ and let $V \in M_d(\mathbb{R})$ be a matrix which fixes $u$ and maps the orthogonal complement of $u$ to zero. We may write
\begin{align*}\langle \mathcal{L}Bu,u\rangle &= \sum_{k=0}^{d-1}\sum_{i_1,\ldots,i_k=1}^M \langle (A_{i_k}\cdots A_{i_1})^TU^TUA_{i_k}\cdots A_{i_1}u,u\rangle\\
&=\sum_{k=0}^{d-1}\sum_{i_1,\ldots,i_k=1}^M \langle UA_{i_k}\cdots A_{i_1}u,UA_{i_k}\cdots A_{i_1}u\rangle\\
&=\sum_{k=0}^{d-1}\sum_{i_1,\ldots,i_k=1}^M \left\|UA_{i_k}\cdots A_{i_1}u\right\|^2
=\sum_{k=0}^{d-1}\sum_{i_1,\ldots,i_k=1}^M \left\|UA_{i_k}\cdots A_{i_1}V\right\|^2>0\end{align*}
by Lemma \ref{le:ire}, and therefore $\mathcal{L}(\mathcal{K}\setminus \{0\})$ is interior to $\mathcal{K}$ as claimed. By a suitable version of the Perron-Frobenius Theorem (for example, the combination of Theorems 3.20 and 3.23 in \cite{BePl94}) it follows that $\rho(L_{\mathsf{A}})$ is nonzero and is an eigenvalue of $L_{\mathsf{A}}$, and that up to scalar multiplication there exists a unique positive definite matrix $Q$ such that $L_{\mathsf{A}}Q=\rho(L_{\mathsf{A}})Q$.

To establish \eqref{eq:prids} we will prove $\rho(L_{\mathsf{A}})\leq \rho(\sum_{i=1}^M A_i^{\otimes 2})\leq e^{P(\mathsf{A},2)}  \leq \rho(L_{\mathsf{A}})$. For the first of these 
we consider the linear map $\mathrm{vec}\colon M_d(\mathbb{R}) \to \mathbb{R}^{d^2}$ obtained by defining the first $d$ entries of $\mathrm{vec}(B)$ to be  the first column of $B$, the second $d$ entries of $\mathrm{vec}(B)$ to be the second column of $B$,  and so forth. We have $\mathrm{vec}(B_1B_2B_3)=(B_3^T\otimes B_1)\mathrm{vec}(B_2)$ for every $B_1,B_2,B_3 \in M_d(\mathbb{R})$ (see e.g. \cite[Lemma 4.3.1]{HoJo94}) and since $L_{\mathsf{A}}Q=\rho(L_{\mathsf{A}})Q$,
\[\left(\sum_{i=1}^M A_i^T\otimes A_i^T\right)\mathrm{vec}(Q)=\rho(L_{\mathsf{A}})\mathrm{vec}(Q).\]
It follows that $\rho(L_{\mathsf{A}})$ is an eigenvalue of $\sum_{i=1}^M (A_i^T)^{\otimes 2}$ and hence of $\sum_{i=1}^M A_i^{\otimes 2}$ which implies the first inequality. For the second inequality we note simply that
\begin{align*}\rho\left(\sum_{i=1}^MA_i^{\otimes 2}\right)&=\lim_{n\to\infty}\left\|\left(\sum_{i=1}^MA_i^{\otimes 2}\right)^n\right\|^{\frac{1}{n}}= \lim_{n\to\infty}\left\|\sum_{i_1,\ldots,i_n=1}^MA_{i_n}^{\otimes 2}\cdots A_{i_1}^{\otimes 2}\right\|^{\frac{1}{n}}\\
&\leq  \lim_{n\to\infty}\left(\sum_{i_1,\ldots,i_n=1}^M\left\|A_{i_n}\cdots A_{i_1}\right\|^2\right)^{\frac{1}{n}}=e^{P(\mathsf{A},2)}\end{align*}
using Gelfand's formula. For the third inequality, we observe that for every positive semidefinite $B \in M_d(\mathbb{R})$ we have $\|B\|^2:=\rho(B^TB)=\rho(B^2)=\rho(B)^2$ and therefore 
\begin{equation}\label{eq:tr}\|B\|=\rho(B)\geq \frac{1}{d}\mathrm{tr}\,B \geq \frac{1}{d}\rho(B)=\frac{1}{d}\|B\|\end{equation}
since all of the eigenvalues of $B$ are non-negative. Since the identity matrix $\mathrm{Id} \in M_d(\mathbb{R})$ is positive definite, Gelfand's formula for $\rho(L_{\mathsf{A}})$ implies
\begin{align*}\rho(L_{\mathsf{A}})\geq \lim_{n \to \infty} \left\|L_{\mathsf{A}}^n {\mathrm{Id}}\right\|^{\frac{1}{n}} &= \lim_{n\to\infty} \left\|\sum_{i_1,\ldots,i_n=1}^M (A_{i_n}\cdots A_{i_1})^TA_{i_n}\cdots A_{i_1}\right\|^{\frac{1}{n}} \\
&= \lim_{n\to\infty} \left(\mathrm{tr}\,\sum_{i_1,\ldots,i_n=1}^M (A_{i_n}\cdots A_{i_1})^TA_{i_n}\cdots A_{i_1}\right)^{\frac{1}{n}} \\
&= \lim_{n\to\infty} \left(\sum_{i_1,\ldots,i_n=1}^M \mathrm{tr}\,(A_{i_n}\cdots A_{i_1})^TA_{i_n}\cdots A_{i_1}\right)^{\frac{1}{n}} \\
&= \lim_{n\to\infty} \left(\sum_{i_1,\ldots,i_n=1}^M \|A_{i_n}\cdots A_{i_1}\|^2\right)^{\frac{1}{n}} =e^{P(\mathsf{A},2)}\end{align*}
by repeated application of \eqref{eq:tr}, and this completes the proof of \eqref{eq:prids}. The existence of solutions $Q,\hat{Q}$ to the equation \eqref{eq:kuid} having been shown, we note that $\mathrm{tr}\,Q\hat{Q}$ is positive: writing $Q=U^TU$ yields
\begin{equation}\label{eq:nonneg}\mathrm{tr}(Q\hat{Q})=\mathrm{tr}(U\hat{Q}U^T)=\sum_{k=1}^d \langle U\hat{Q}U^Te_k,e_k\rangle=\sum_{k=1}^d \langle \hat{Q}U^Te_k,U^Te_k\rangle >0\end{equation}
since $\hat{Q}$ is positive definite. It follows that by replacing $Q$ with $(\mathrm{tr}(Q\hat{Q}))^{-1}Q$ if necessary we may obtain the desired normalisation $\mathrm{tr}(Q\hat{Q})=1$.
\end{proof}

We may now recall the definition of the \emph{Kusuoka measure} $\mu$ associated to the matrices $(A_1,\ldots,A_M)$: if $Q,\hat{Q}$ are the matrices provided by Proposition \ref{pr:ku}, then we may define the measure $\mu$ on cylinders of $\Sigma_M$ by 
\[\mu\left(\left[x_1\cdots x_n\right]\right):=e^{-nP(\mathsf{A},2)}\mathrm{tr}\,\left(\hat{Q}(A_{x_n}\cdots A_{x_1})^TQA_{x_n}\cdots A_{x_1}\right)\]
for every $x_1,\ldots,x_n \in \{1,\ldots,M\}$. Since this is the trace of the product of the positive definite matrix $\hat{Q}$ with a positive semidefinite matrix of the form $B^TQB$, it follows as in \eqref{eq:nonneg} that it is non-negative. We note that
\[\sum_{k=1}^M \mu\left(\left[kx_2\cdots x_n\right]\right)=\mu\left(\left[x_2\cdots x_n\right]\right), \]
\[\sum_{k=1}^M \mu\left(\left[x_1\cdots x_{n-1}k\right]\right)=\mu\left(\left[x_1\cdots x_{n-1}\right]\right) \]
using the equations \eqref{eq:kuid}, and $\mu(\sigma_M)=1$ by the normalisation condition $\mathrm{tr}(Q\hat{Q})=1$. It follows that $\mu$ describes a $\sigma$-invariant probability measure on the ring of finite unions of cylinder sets and hence defines a $\sigma$-invariant probability measure on $\Sigma_M$. The purpose of this appendix is to note the following:
\begin{theorem}
Let $\mathsf{A}=(A_1,\ldots,A_M) \in M_d(\mathbb{R})^M$ be irreducible, where $M \geq 2$, $d \geq 1$. Then the Kusuoka measure associated to $\mathsf{A}$ is precisely the equilibrium state of $(\mathsf{A},2)$.
\end{theorem}
\begin{proof}
Let $Q$, $\hat{Q}$ be as given by Proposition \ref{pr:ku} and write $Q=U^TU$, $Q=\hat{U}^T\hat{U}$ where $U,\hat{U}$ are real upper triangular matrices with positive diagonal entries. Let $\mu$ denote the Kusuoka measure associated to $\mathsf{A}$ and let $\nu$ denote the equilibrium state of $(\mathsf{A},2)$ constructed by Theorem \ref{th:Gibbs}. Recall that the \emph{Frobenius norm} of a matrix $B \in M_d(\mathbb{R})$ is given by $\|B\|_F^2:=\mathrm{tr}\,B^TB$ and satisfies $\|B\|\leq \|B\|_F \leq \sqrt{d}\cdot \|B\|$. For every $x_1,\ldots,x_n \in \{1,\ldots,M\}$ we have
\begin{align*}\mathrm{tr}\,\left(\hat{Q}(A_{x_n}\cdots A_{x_1})^TQA_{x_n}\cdots A_{x_1}\right) &= \mathrm{tr}\,\left(\hat{U}^T\hat{U}(A_{x_n}\cdots A_{x_1})^TU^TUA_{x_n}\cdots A_{x_1}\right)\\
&=\left\|UA_{x_n}\cdots A_{x_1}\hat{U}^T\right\|_F^2\end{align*}
so that by the Gibbs inequality \eqref{eq:Gibbs}
\begin{align*}\frac{\mu([x_1\cdots x_n])}{e^{-nP(\mathsf{A},2)}} =\left\|UA_{x_n}\cdots A_{x_1}\hat{U}^T\right\|_F^2 &\leq d\|U\|^2 \|\hat{U}\|^2 \|A_{x_n}\cdots A_{x_1}\|^2\\
&\leq \frac{Cd\|U\|^2\|\hat{U}\|^2 \nu([x_1\cdots x_n])}{e^{-nP(\mathsf{A},2)}}.\end{align*}
It follows that the invariant probability measure $\mu$ is absolutely continuous with respect to the ergodic probability measure $\nu$ and therefore $\mu=\nu$ as claimed.
\end{proof}
We therefore note that all of the results concerning matrix equilibrium states in this article also hold true for Kusuoka measures.

The mixing properties of Kusuoka measures have recently been investigated by Johansson, {\"O}berg and Pollicott, who showed in \cite{JoObPo15} that if the $M$-tuple of linear transformations of $M_d^S(\mathbb{R})$ defined by $B \mapsto A_i^TBA_i$ is irreducible then the Kusuoka measure associated to $\mathsf{A}$ is mixing, and moreover satisfies a quantitative mixing estimate of the form
\[\left|\mu\left([x_1\cdots x_m] \cap \sigma^{-n}[y_1\cdots y_m]\right) -\mu\left([x_1\cdots x_m]\right)\mu\left([y_1\cdots y_m]\right) \right|=O(e^{-\gamma n})\]
where $\gamma>0$ depends only on $\mathsf{A}$. It is a simple matter to modify Proposition \ref{pr:notmix} to produce examples of Kusuoka measures  not satisfying the irreducibility condition which are not mixing. For example, if $\alpha_1,\alpha_2>0$ are unequal and satisfy $\alpha_1^2+\alpha_2^2=1$ then the matrices
\[A_1:=\left(\begin{array}{cc}0&\alpha_2 \\ \alpha_1 &0\end{array} \right),\qquad A_2:=\left(\begin{array}{cc}0&\alpha_1\\ \alpha_2&0\end{array} \right)\]
form an irreducible pair and satisfy
\[\sum_{i=1}^2 A_i^T\mathrm{Id}A_i =\sum_{i=1}^2 A_i\mathrm{Id}A_i^T = \mathrm{Id},\]
and so the Kusuoka measure associated to $\mathsf{A}:=(A_1,A_2)$ is given simply by
\[\mu([x_1\cdots x_n])=\frac{1}{2}\mathrm{tr}\left((A_{x_n}\cdots A_{x_1})^TA_{x_n}\cdots A_{x_1}\right)\]
and it may be shown by direct calculation that $\mu$ is the barycentre of two distinct $\sigma^2$-invariant Bernoulli measures. We omit the details.

\bibliographystyle{acm}
\bibliography{pradjrad}

\end{document}